\documentclass[final,leqno]{siamltex}

\usepackage{amsfonts,amssymb} 
\usepackage{amsmath} 
\usepackage{color} 
\usepackage{graphicx} 
\usepackage{epsfig,mathrsfs} 
\usepackage[bf,SL,BF]{subfigure} 
\usepackage{fancyhdr} 
\usepackage{CJK} 
\usepackage{caption} 
\usepackage{wrapfig} 
\usepackage{cases} 
\usepackage{bm}
\usepackage{algorithm}
\usepackage{algorithmic}
\usepackage{enumerate}
\newtheorem{remark}{Remark}[section] 
\newtheorem{example}{Example}[section] 

\title{High-order BDF fully discrete scheme for backward fractional Feynman-Kac equation with nonsmooth data 
\thanks{This work was supported by the National Natural Science Foundation of China under Grant No. 11671182, and the AI and Big Data Funds under Grant No. 2019620005000775.
}}

\author{Jing Sun\footnotemark[2] \and Daxin Nie\footnotemark[2]
\and Weihua Deng\footnotemark[2]\thanks{ School of Mathematics and Statistics, Gansu Key Laboratory of Applied Mathematics and Complex Systems, Lanzhou University, Lanzhou 730000, P.R. China (Email: dengwh@lzu.edu.cn).}
}
\begin{document}

\maketitle

\begin{abstract}
 The Feynman-Kac equation governs the distribution of the statistical observable --- functional,  having wide applications in almost all disciplines. After overcoming challenges from the time-space coupled nonlocal operator and the possible low regularity of functional, this paper develops the high-order fully discrete scheme for the backward fractional Feynman-Kac equation by using  backward difference formulas (BDF) convolution quadrature in time, finite element method in space, and some correction terms. With a systematic correction, the high convergence order is achieved up to $6$ in time, without deteriorating the optimal convergence in space and without the regularity requirement on the solution. Finally, the extensive numerical experiments validate the effectiveness of the high-order schemes.

%
%
\end{abstract}
\begin{keywords}
fractional substantial derivative,  time-space coupled operator, finite element method,  higher-order backward difference formulas,  convolution quadrature,  error analysis.
\end{keywords}

\begin{AMS}
35R11, 65M60, 65M12, 65F08
\end{AMS}

\pagestyle{myheadings}
\thispagestyle{plain}
\markboth{J. Sun, D. X. Nie and W. H. Deng}{The modified higher-order BDF scheme for backward fractional Feynman-Kac equation}

\noindent{\section{Introduction}}
Functional is a class of statistical observables, having wide applications in almost all disciplines \cite{Deng2020}. Anomalous diffusions are ubiquitous in the nature world. The distribution of functionals for anomalous diffusion is governed by fractional Feynman-Kac equation \cite{Carmi2010,Deng2020,TurgemanCarmiBarkai:09}. From the point view of practical applications, one has to resort to numerical method to get the solution of the fractional Feynman-Kac equation; there are already some works on this issue; see, e.g., \cite{Chenminghua:15,ChenDeng:18,DengLiQianWang:18,ZhangDeng:17}, but most of the numerical analyses of which need the regularity assumption on the exact solution and the functional.
This paper provides the high-order schemes for the fractional Feynman-Kac equation without the regularity requirements on both the exact solution and the functional.


In the past few years, the high-order algorithms have been proposed extensively for fractional diffusion equation \cite{chen2015,ford2017,jin2017,li:2020,yan2018}, the nonlocal property of which makes them keep the same computational cost but greatly improve the accuracy. As for fractional Feynman-Kac equation, there are less relative discussions; \cite{Chenminghua:15} studies the high-order scheme to solve backward fractional Feynman-Kac equation with truncated L{\'e}vy flights, and the detailed error and stability analyses for the $1$-st order scheme are provided with some regularity assumptions on the solution; \cite{ChenDeng:18} provides a high-order scheme for the time tempered fractional Feynman-Kac equation when the solution $G(x_{0},\rho,t)\in C^{2}(\Omega)$. In practice, the regularity assumptions on the solution and functional are usually hard to be satisfied. So the robust numerical algorithms without any regularity assumptions are more practicable.

Recently, \cite{Sun:2020} provides $1$-st and $2$-nd order schemes to solve homogeneous backward fractional Feynman-Kac equation, and the relative error estimates without regularity assumptions are established, but higher-order scheme of this approach for backward fractional Feynman-Kac equation is still not available and there are many challenges in deriving the error estimates 
for the case with nonsmooth data. Here, under the condition of ensuring the spatial accuracy, we provide a systematic $k$-th order $(k=1,2,\ldots,6)$ fully discrete scheme with a complete error analysis for inhomogeneous backward fractional Feynman-Kac equation \cite{Carmi2010, Deng2020}, i.e.,
\begin{equation}\label{equaty}
\left\{
\begin{aligned}
&\frac{\partial G(x_0,\rho,t)}{\partial t}=\,_0D^{1-\alpha,x_0}_t\Delta G(x_0,\rho,t)\\
&\qquad \qquad \qquad\qquad -\rho U(x_0)G(x_0,\rho,t)+f(x_0,\rho,t), \qquad(x_0,t)\in \Omega\times(0,T],\\
&G(x_0,\rho,0)=G_0(x_0),\qquad\qquad\qquad\qquad\qquad\qquad\qquad\qquad  x_0\in \Omega,\\
&G(x_0,\rho,t)=0,\qquad\qquad\quad\qquad\qquad\qquad\qquad\qquad\qquad~~~ (x_0,t)\in \partial \Omega\times(0,T],
\end{aligned}
\right.
\end{equation}
where $G(x_0,\rho,t)=\int_{-\infty}^{\infty}G(x_0,\mathbb{A},t)e^{-\mathbf{i}\rho \mathbb{A}}d \mathbb{A}$ with $\mathbf{i}$ being the imaginary unit; $G(x_0, \mathbb{A}, t)$ is the joint probability density function of finding the particle with the functional $\mathbb{A}$ at time $t$ and the initial position of the particle at $x_0$; the functional 
$\mathbb{A}=\int_{0}^{t}U(x_{0}(\tau))d\tau$ with $U(x_0)$ being a prescribed function depending on the concrete applications \cite{Deng2020,Kac:1949} and $x_0(t)$ a trajectory of anomalous diffusion starting at $x_{0}$; $\Delta$ means the Laplace operator; $\alpha\in(0,1)$; $f(x_0,\rho,t)$ is the source term; $\Omega$ is a bounded convex polygonal domain in $\mathbb{R}^{n}$ $(n=1,2,3)$ and we assume that $U(x_{0})$ is bounded in $\bar{\Omega}$ in this paper; $T$ is a fixed final time; $~_{0}D^{\alpha,x_0}_t$ denotes Riemann-Liouville fractional substantial derivative defined by \cite{LiDengZhao:19}
\begin{equation}\label{eqRLFD}
\begin{aligned}
~_{0}D^{\alpha,x_0}_tG(x_0,\rho,t)=e^{-t\rho U(x_0)}~_0D^{\alpha}_t(e^{t\rho U(x_0)}G(x_0,\rho,t)), \qquad \alpha\in(0,1),
\end{aligned}
\end{equation}
and $~_{0}D^{\alpha}_{t}$ means the Riemann-Liouville fractional derivative with its definition \cite{Podlubny1999}
\begin{equation*}
~_{0}D^{\alpha}_{t}G(x_0,\rho,t)=\frac{1}{\Gamma(1-\alpha)}\frac{\partial}{\partial t}\int^t_{0}(t-\xi)^{-\alpha}G(x_0,\rho,\xi)d\xi, \qquad \alpha\in(0,1).
\end{equation*}

Following \cite{ZhangDeng:17} and using the relationship between the Caputo and Riemann-Liouville fractional derivatives, one can get the equivalent form of Eq. \eqref{equaty}, i.e.,
\begin{equation}\label{eqretosol}
\left\{
\begin{aligned}
& \,_0D^{\alpha,x_0}_tG(x_0,\rho,t)-\Delta G(x_0,\rho,t)\\
& ~~ =e^{-\rho U(x_{0})t}\,_0D^{\alpha}_tG(x_0,\rho,0)+\,_{0}I^{1-\alpha,x_{0}}_{t}f(x_0,\rho,t), \quad(x_0,t)\in \Omega\times(0,T],\\
&G(x_0,\rho,0)=G_0(x_0),\qquad\qquad\qquad\qquad\qquad\qquad\qquad  x_0\in \Omega,\\
&G(x_0,\rho,t)=0,\qquad\qquad\quad\qquad\qquad\qquad\qquad\qquad\quad\,\, (x_0,t)\in \partial\Omega\times(0,T],
\end{aligned}
\right.
\end{equation}
where  $~_{0}I^{\alpha,x_{0}}_{t}$ means Riemann-Liouville fractional substantial integral defined by \cite{LiDengZhao:19}
\begin{equation*}
\begin{aligned}
~_{0}I^{\alpha,x_0}_tf=\frac{1}{\Gamma(\alpha)}\int_{0}^{t}(t-\tau)^{\alpha-1}e^{-\rho U(x_{0})(t-s)}f(s)ds,\qquad \alpha\in(0,1).
\end{aligned}
\end{equation*}
Compared \eqref{eqretosol} with \eqref{equaty}, we separate the operators $_{0}D^{1-\alpha,x_{0}}_{t}$ and $-\Delta$ to reduce the influences of the regularity of $U(x_{0})$ on convergence order in space, it is more effective to establish the numerical scheme based on \eqref{eqretosol} instead of \eqref{equaty}.

The correction scheme of higher-order BDF convolution quadrature for fractional evolution equation is provided in \cite{jin2017}. If its idea is applied to solve Eq. \eqref{eqretosol}, it may deteriorate the convergence order in space.
The main reason is that the Riemann-Liouville fractional substantial derivative is a time-space coupled nonlocal operator, which makes $\beta(z,x_{0})$ and $L^{2}$ projection $P_{h}$ non-commutable (one can refer to Secs. $2$ and $3$ for the definitions of the operators). 
Thus, to preserve the optimal convergence rates in space, we build the finite element scheme by applying the $L^{2}$ projection operators  $P_{h}$ on $e^{-\rho U(x_{0})t}~_{0}D^{\alpha}_{t}G_{0}$ and $~_{0}I^{1-\alpha,x_{0}}_{t}f$ instead of $G_{0}$ and $f$,
which avoids estimating the errors aroused by $\|((\beta_{\tau,k}(z))^{\alpha}+A_{h})^{-1}P_{h}((\beta_{\tau,k}(z))^{\alpha-1}\mu_{k}(e^{-\beta(z)\tau})G_{0})-((\beta_{\tau,k}(z,x_{0}))^{\alpha}+A_{h})^{-1}(\beta_{\tau,k}(z))^{\alpha-1}\mu_{k}(e^{-\beta(z)\tau})P_{h}G_{0}\|_{L^{2}(\Omega)}$,  and the regularity requirement on $U(x_{0})$ is also relaxed naturally. Moreover, according to the finite element scheme \eqref{eqfinsch} (see Sec. 3), we provide novel high-order fully discrete schemes, which can achieve $k$-th order convergence in time and optimal convergence in space without any regularity assumptions on exact solution.

The rest of the paper is organized as follows. We first provide some preliminaries and a regularity estimate for the solution of Eq. \eqref{eqretosol} in Sec. 2.  In Sec. 3, we use the finite element method to discretize the Laplace operator and provide novel high-order approximations in time based on the high-order BDF convolution quadrature. The error analyses presented in Sec. 4 show that our schemes can not only preserve $k$-th order convergence in time, but also achieve optimal convergence rate in space.
In Sec. 5, extensive numerical experiments are performed to show the effectiveness of the schemes.
We conclude the paper with some discussions in the last section. Throughout the paper, the generic constant $C>0$ may be different at different occurrences and $\epsilon>0$ is an arbitrarily small constant.

\section{Preliminaries}

We set $A=-\Delta$ with a zero Dirichlet condition in the following. For any $ q\geq  0 $, denote the space $ \dot{H}^{q}(\Omega)=\{v\in L^2(\Omega): \|v\|^2_{\dot{H}^q(\Omega)}<\infty \}$ with the norm \cite{Thomee2006}
\begin{equation*}
\|v\|^2_{\dot{H}^q(\Omega)}=\sum_{j=1}^{\infty}\lambda_j^q(v,\varphi_j)^2,
\end{equation*}
where $ {(\lambda_j,\varphi_j)} $ are the eigenvalues ordered non-decreasingly and the corresponding eigenfunctions (normalized in the $ L^2(\Omega) $ norm) of operator $A$.

Below  we define sectors $\Sigma_{\theta}$ and $\Sigma_{\theta,\kappa}$ in the complex plane $\mathbb{C}$, i.e., for $\kappa>0$ and $\pi/2<\theta<\pi$,
\begin{equation*}
\begin{aligned}
&\Sigma_{\theta}=\{z\in\mathbb{C}\setminus \{0\},|\arg z|\leq \theta\}, \\
&\Sigma_{\theta,\kappa}=\{z\in\mathbb{C}:|z|\geq\kappa,|\arg z|\leq \theta\},\\
\end{aligned}
\end{equation*}
and the contour $\Gamma_{\theta,\kappa}$ is defined by
\begin{equation*}
\Gamma_{\theta,\kappa}=\{z\in\mathbb{C}: |z|=\kappa,|\arg z|\leq \theta\}\cup\{z\in\mathbb{C}: z=r e^{\pm \mathbf{i}\theta}: r\geq \kappa\},
\end{equation*}
oriented with an increasing imaginary part, where $\mathbf{i}$ denotes the imaginary unit and $\mathbf{i}^2=-1$. Then we denote  $\|\cdot\|$ as the operator norm from $L^2(\Omega)$ to $L^2(\Omega)$, and use the notation `$\widetilde{u}$' as the Laplace transform of $u$ and the abbreviations $G(t)$, $G_0$ and $f$ for $G(x_0,\rho,t)$, $G_0(x_0)$ and $f(x_0,\rho,t)$ respectively in the following.

According to the Laplace transforms of Riemann-Liouville fractional substantial derivative and integral \cite{LiDengZhao:19}, the Laplace transform representation of Eq. \eqref{eqretosol} can be given as
\begin{equation}\label{equsolrep}
\tilde{G}(z)=((\beta(z,x_0))^{\alpha}+A)^{-1}(\beta(z,x_0))^{\alpha-1}G_{0}+((\beta(z,x_0))^{\alpha}+A)^{-1}(\beta(z,x_0))^{\alpha-1}\tilde{f},
\end{equation}
where
\begin{equation}\label{defbeta}
\beta(z,x_0)=z+\rho U(x_0)
\end{equation}
and we denote it briefly by $\beta(z)$ below. Then we present an estimate of $\beta(z)$ and the regularity estimate of the solution of Eq. \eqref{eqretosol}.
\begin{lemma}[\cite{DengLiQianWang:18}]\label{lemmaBeta}
	Let $\beta(z)$ be defined in \eqref{defbeta} and $U(x_0)$ is bounded in $\bar{\Omega}$. By choosing $\theta \in \left(\frac{\pi}{2},\pi\right)$ sufficiently close to $\frac{\pi}{2}$
	and $\kappa>0$ sufficiently large (depending on the value $|{{\rho}}|\|U(x_0)\|_{L^{\infty}(\bar{\Omega})}$), we have 
	\begin{enumerate}[(1)]
		\item For all $x_{0}\in \bar{\Omega}$ and ${{z}}\in \Sigma_{\theta,\kappa}$, it holds that $\beta({{z}}) \in \Sigma_{\frac{3\pi}{4},\frac{\kappa}{2}}$ and
		\begin{equation}\label{chapter4section2_2prop1conc1}
		C_1|{{z}}|\leq|\beta({{z}})|\leq C_2|{{z}}|,
		\end{equation}
		where $C_1$ and $C_2$ denote two positive constants.
		So $\beta({{z}})^{1-\alpha}$ and $\beta({{z}})^{\alpha-1}$ are both   analytic function of ${{z}}\in \Sigma_{\theta,\kappa} $.
		
		\item The operator $((\beta({{z}}))^\alpha+A)^{-1}:L^2(\Omega)\rightarrow L^2(\Omega)$ is well-defined, bounded, and analytic with respect to $z\in \Sigma_{\theta,\kappa}$, satisfying
		\begin{equation}\label{chapter4section2_2prop1conc21}
		\|A((\beta(z))^\alpha+A)^{-1}\|\leq C~~~~~ \forall{{z}} \in \Sigma_{\theta,\kappa},
		\end{equation}
		and
		\begin{equation}\label{chapter4section2_2prop1conc22}
		\|((\beta({{z}}))^\alpha+A)^{-1}\|\leq C|{{z}}|^{-\alpha}~~~\forall{{z}} \in \Sigma_{\theta,\kappa}.
		\end{equation}
	\end{enumerate}
\end{lemma}
Combining \eqref{equsolrep} and Lemma \ref{lemmaBeta}, one can get the regularity estimate for the solution of Eq. \eqref{eqretosol} (refer to \cite{Sun:2020} for the detailed proof).
\begin{theorem}\label{thmreg}
	Let $G(t)$ be the solution of Eq. \eqref{eqretosol}. Assume $U(x_0)$ is bounded in $\bar{\Omega}$. If $G_{0}\in L^{2}(\Omega)$ and $\int_{0}^{t}(t-s)^{-\sigma\alpha/2}\|f(s)\|_{L^2(\Omega)}ds< \infty$, then we have the estimate
	\begin{equation*}
	\|G(t)\|_{\dot{H}^{\sigma}(\Omega)}\leq Ct^{-\sigma\alpha/2}\|G_0\|_{L^2(\Omega)}+C\int_{0}^{t}(t-s)^{-\sigma\alpha/2}\|f(s)\|_{L^2(\Omega)}ds, \quad \sigma\in[0,2].
	\end{equation*}
\end{theorem}

\section{Modified high-order BDF fully discrete scheme}
In this section, we first use finite element method to discretize Laplace operator in \eqref{eqretosol}. Then the modified high-order BDF fully discrete scheme is constructed based on finite element semi-discrete scheme and the corresponding correction criteria are also proposed.

 Let $\mathcal{T}_h$ be a shape regular quasi-uniform partitions of the domain $\Omega$,  where $h$ is the maximum diameter. Denote $ X_h $ as piecewise linear finite element space
\begin{equation*}
	X_{h}=\{v_h\in C(\bar{\Omega}): v_h|_\mathbf{T}\in \mathcal{P}^1,\  \forall \mathbf{T}\in\mathcal{T}_h,\ v_h|_{\partial \Omega}=0\},
\end{equation*}
where $\mathcal{P}^1$ denotes the set of piecewise polynomials of degree $1$ over $\mathcal{T}_h$. We denote by $(\cdot,\cdot)$ $L^{2}$ inner product and define the $ L^2 $-orthogonal projection $ P_h: L^2(\Omega)\rightarrow X_h $ by
\begin{equation*}
\begin{aligned}
&(P_hu,v_h)=(u,v_h) ~~~~\forall v_h\in X_h.\\
\end{aligned}
\end{equation*}
Then we use finite element method to discretize the operator $-\Delta$ and the finite element scheme of Eq. \eqref{eqretosol} can be written as: Find $G_{h}(t)\in X_{h}$ such that
\begin{equation}\label{eqfinsch}
	\begin{aligned}
		&(\,_0D^{\alpha,x_0}_tG_{h},v_{h})+(\nabla G_{h},\nabla v_{h}) =\\
		&\qquad\qquad\qquad(e^{-\rho U(x_{0})t}\,_0D^{\alpha}_tG(x_0,\rho,0),v_{h})+(\,_{0}I^{1-\alpha,x_{0}}_{t}f(x_0,\rho,t),v_{h})\qquad \forall v_{h}\in X_{h}.
	\end{aligned}
\end{equation}

Different from the traditional finite element scheme, we apply the $L^2$ projection $P_{h}$ on $e^{-\rho U(x_{0})t}\,_0D^{\alpha}_tG(x_0,\rho,0)$ and $\,_{0}I^{1-\alpha,x_{0}}_{t}f(x_0,\rho,t)$ instead of $G_{0}$ and $f$. Thus the errors between $P_{h}(e^{-\rho U(x_{0})t}\,_0D^{\alpha}_tG(x_0,\rho,0))$ and $e^{-\rho U(x_{0})t}\,_0D^{\alpha}_tP_{h}(G(x_0,\rho,0))$ and the ones between $P_{h}(\,_{0}I^{1-\alpha,x_{0}}_{t}f(x_0,\rho,t))$ and $\,_{0}I^{1-\alpha,x_{0}}_{t}P_{h}(f(x_0,\rho,t))$ are no longer needed to be considered, which relaxes the regularity requirement on $U(x_{0})$. See the relative error analyses in Sec. 4 below.

 Next, we present the modified high-order BDF fully discrete scheme in detail. Let the time step size $\tau=T/N$ with $N\in\mathbb{N}$, $t_i=i\tau$, $i=0,1,\ldots,N$, and $0=t_0<t_1<\cdots<t_N=T$. Introduce the generating function $\delta_{\tau,k}(\zeta)$ \cite{lubich1988-1,lubich1988-2,lubich1996} and $\beta_{\tau,k}(z)$ $(k=1,2,\ldots,6)$ as
\begin{equation}\label{betatauOk}
\delta_{\tau,k}(\zeta)=\frac{1}{\tau}\delta_{k}(\zeta)=\frac{1}{\tau}\sum_{i=1}^{k}\frac{(1-\zeta)^{i}}{i!},\quad \beta_{\tau,k}(z)=\delta_{\tau,k}(e^{-\tau\beta(z)}),
\end{equation}
where $\beta(z)$ is defined in \eqref{defbeta}.
And $\delta_{k}(\zeta)$ has the following property.
\begin{lemma}[\cite{lubich1988-1}]\label{lemdk}
	$\delta_{k}(\zeta)$ is analytic and without zeros in a neighborhood of the closed unit disc $|\zeta|\leq 1$, with the exception of a zero at $\zeta=1$, and $\delta_{k}(\zeta)$ satisfies that
	\begin{equation*}
	|\arg \delta_{k}(\zeta)|\leq \pi -\vartheta_{k}\quad {\rm for}~|\zeta|<1,
	\end{equation*}
	where $\vartheta_{k}=90^{\circ}$, $90^{\circ}$, $88^{\circ}$, $73^{\circ}$, $51^{\circ}$, $18^{\circ}$ for $k=1,\ldots,6$, respectively.
\end{lemma}

Generally, according to the convolution quadrature \cite{jin2017,lubich1988-1,lubich1988-2,lubich1996} generated by $k$-th order BDF, the Riemann-Liouville fractional derivative with $\alpha \in (0,1)$ can be approximated by
\begin{equation*}
~_{0}D^{\alpha}_{t}\varphi(t_{n})\approx \sum_{i=0}^{n}d^{\alpha,k}_{i}\varphi^{n-i},
\end{equation*}
where $\varphi^{n}=\varphi(t_{n})$ and
\begin{equation*}
	(\delta_{\tau,k}(\zeta))^{\alpha}=\sum_{i=0}^{\infty}d^{\alpha,k}_{i}\zeta^{i}.
\end{equation*}
 Similarly, the Riemann-Liouville fractional substantial derivative can be approximated by
\begin{equation}\label{eqdisstfd}
~_{0}D^{\alpha,x_{0}}_{t}\varphi(t_{n})\approx \sum_{i=0}^{n}d^{\alpha,k}_{i}e^{-\rho U(x_{0})t_{i}}\varphi^{n-i}.
\end{equation}

 By using \eqref{eqdisstfd}, we have the following $k$-th order BDF fully discrete scheme: Find $G^{n}_{h}\in X_{h}$ such that
\begin{equation}\label{equncorrect}
	\begin{aligned}
		&\sum_{i=0}^{n-1}d^{\alpha,k}_{i}(e^{-t_{i}\rho U(x_{0})}G^{n-i}_{h} ,v_{h})+(\nabla G^{n}_{h}, \nabla v_{h})\\
		&\qquad=\sum_{i=0}^{n-1}d^{\alpha,k}_{i}(e^{-t_{n}\rho U(x_{0})}G^{0},v_{h})+\sum_{i=0}^{n-1}d^{\alpha-1,k}_{i}(e^{-t_{i}\rho U(x_{0})}f^{n-i},v_{h}),
	\end{aligned}
\end{equation}
where $f^{n}=f(t_{n})$ and $G^{0}=G_{0}(x_{0})$. In fact, for \eqref{equncorrect}, the desired $k$-th order accuracy can be reached only under the condition that the solution is regular enough.
So here, we try to modify the scheme and get a robust $k$-th order scheme for the case with nonsmooth data. First, by Taylor's expansion, we spilt $f$ into
\begin{equation}\label{eqtaylorf}
f(t)=\sum_{i=0}^{k-2}\frac{t^{i}}{i!}\partial^{i}_{t}f(0)+R_{k}(t),
\end{equation}
where
\begin{equation*}
R_{k}(t)=\frac{t^{k-1}}{(k-1)!}\partial^{k-1}_{t}f(0)+\frac{t^{k-1}}{(k-1)!}\ast\partial^{k}_{t}f(t)
\end{equation*}
and `$\ast$' denotes the convolution. To capture the regularity property of the solution at starting point, the $k$-th order BDF fully discrete scheme can be modified as: Find $G^{n}_{h}\in X_{h}$ such that

\begin{equation}\label{eqfullschemes1}
\begin{aligned}
&\sum_{i=0}^{n-1}d^{\alpha,k}_{i}(e^{-t_{i}\rho U(x_{0})}G^{n-i}_{h} ,v_{h})+(\nabla G^{n}_{h}, \nabla v_{h})-\sum_{j=1}^{k-1}d^{\alpha,k}_{n-j}a^{(k)}_{j}(e^{-t_{n}\rho U(x_{0})}G^{0},v_{h})\\
&\qquad=\sum_{i=0}^{n-1}d^{\alpha,k}_{i}(e^{-t_{n}\rho U(x_{0})}G^{0},v_{h})+\sum_{i=0}^{n-1}d^{\alpha-1,k}_{i}(e^{-t_{i}\rho U(x_{0})}f^{n-i},v_{h})\\
&\qquad\quad+\sum_{j=1}^{k-1}a^{(k)}_{j}d^{\alpha-1,k}_{n-j}(e^{-t_{n-j}\rho U(x_{0})}f^{0},v_{h})\\
&\qquad\quad+\sum_{l=1}^{k-2}\sum_{j=1}^{k-1}b^{(k)}_{l,j}\tau^{l}d^{\alpha-1,k}_{n-j}(e^{-t_{n-j}\rho U(x_{0})}\partial^{l}_{t}f(0),v_{h})
 \qquad\qquad \forall v_{h}\in X_{h},
\end{aligned}
\end{equation}
where $a^{(k)}_{j}$ and $b^{(k)}_{l,j}$ are coefficients to be determined below.
Next, introduce $A_{h}$ as
\begin{equation*}
	(A_{h}u_{h},v_{h})=(\nabla u_{h},\nabla v_{h})\qquad \forall u_{h},v_{h}\in X_{h}.
\end{equation*}
Thus \eqref{eqfullschemes1} can be expressed as
\begin{equation}\label{eqfullscheme}
\begin{aligned}
&\sum_{i=0}^{n-1}d^{\alpha,k}_{i}(e^{-t_{i}\rho U(x_{0})}G^{n-i}_{h}-P_{h}(e^{-t_{n}\rho U(x_{0})}G^{0}))+A_{h}G^{n}_{h}\\
&=\sum_{j=1}^{k-1}d^{\alpha,k}_{n-j}a^{(k)}_{j}P_{h}(e^{-t_{n}\rho U(x_{0})}G^{0})+\sum_{i=0}^{n-1}d^{\alpha-1,k}_{i}P_{h}(e^{-t_{i}\rho U(x_{0})}f^{n-i})\\
&\quad +\sum_{j=1}^{k-1}a^{(k)}_{j}d^{\alpha-1,k}_{n-j}P_{h}(e^{-t_{n-j}\rho U(x_{0})}f^{0})\\
&\quad
+\sum_{l=1}^{k-2}\sum_{j=1}^{k-1}b^{(k)}_{l,j}\tau^{l}d^{\alpha-1,k}_{n-j}P_{h}(e^{-t_{n-j}\rho U(x_{0})}\partial^{l}_{t}f(0)).
\end{aligned}
\end{equation}

\begin{theorem} \label{thmfulldis}
	The solution of fully discrete scheme \eqref{eqfullscheme} can be represented as
	\begin{equation}\label{eqfulldissol}
	\begin{aligned}
&	G^{n}_{h}= \\
&\frac{1}{2\pi \mathbf{i}}\int_{\Gamma^{\tau}_{\theta,\kappa}}e^{zt_{n}}((\beta_{\tau,k}(z))^{\alpha}+A_{h})^{-1}P_{h}\left ((\beta_{\tau,k}(z))^{\alpha-1}\mu_{k}(e^{-\beta(z)\tau})G^{0}\right )dz\\
	&+\frac{1}{2\pi \mathbf{i}}\int_{\Gamma^{\tau}_{\theta,\kappa}}e^{zt_{n}}((\beta_{\tau,k}(z))^{\alpha}+A_{h})^{-1}P_{h}\left ((\beta_{\tau,k}(z))^{\alpha-1}(\delta_{\tau,k}(e^{-z\tau}))^{-1}\mu_{k}(e^{-z\tau})f^{0}\right )dz\\
	&+\frac{1}{2\pi \mathbf{i}}\int_{\Gamma^{\tau}_{\theta,\kappa}}e^{zt_{n}}((\beta_{\tau,k}(z))^{\alpha}+A_{h})^{-1}\sum_{l=1}^{k-2}P_{h}\left ((\beta_{\tau,k}(z))^{\alpha-1}\eta_{k,l}(e^{-z\tau})\partial^{l}_{t}f(0)\right )dz\\
	&+\frac{1}{2\pi \mathbf{i}}\int_{\Gamma^{\tau}_{\theta,\kappa}}e^{zt_{n}}((\beta_{\tau,k}(z))^{\alpha}+A_{h})^{-1}P_{h}\left ((\beta_{\tau,k}(z))^{\alpha-1}\tau\sum_{n=1}^{\infty}R^{n}_{k}e^{-zt_{n}}\right )dz\\
	\end{aligned}
	\end{equation}
	with the contour $\Gamma^\tau_{\theta,\kappa}=\{z\in \mathbb{C}:\kappa\leq |z|\leq\frac{\pi}{\tau\sin(\theta)},|\arg z|=\theta\}\cup\{z\in \mathbb{C}:|z|=\kappa,|\arg z|\leq\theta\}$ and
	\begin{equation*}
	\begin{aligned}
		&\mu_{k}(\zeta)=\delta_{k}(\zeta)\left (\frac{\zeta}{1-\zeta}+\sum_{j=1}^{k-1}a_{j}^{(k)}\zeta^{j}\right ),\qquad \gamma_{l}(\zeta)=\left (\zeta \frac{d}{d\zeta}\right )^{l}\frac{1}{1-\zeta},\\
		&\eta_{k,l}(\zeta)=\left(\frac{\gamma_{l}(\zeta)}{l!}+\sum_{j=1}^{k-1}b^{(k)}_{l,j}\zeta^{j}\right)\tau^{l+1}.
	\end{aligned}
	\end{equation*}
Here $R^{n}_{k}=R_{k}(t_{n})$.
\end{theorem}
\begin{proof}
	Multiplying $\zeta^{n}$ on both sides of $\eqref{eqfullscheme}$ and summing $n$ from $1$ to $\infty$ yield
	\begin{equation*}
	\begin{aligned}
	&\sum_{n=1}^{\infty}\sum_{i=0}^{n-1}d^{\alpha,k}_{i}e^{-t_{i}\rho U(x_{0})}G^{n-i}_{h}\zeta^{n}+\sum_{n=1}^{\infty}A_{h}G^{n}_{h}\zeta^{n}\\
	&=	\sum_{n=1}^{\infty}\sum_{i=0}^{n-1}d^{\alpha,k}_{i}P_{h}(e^{-t_{n}\rho U(x_{0})}G^{0})\zeta^{n}+\sum_{n=1}^{\infty}\sum_{j=1}^{k-1}d^{\alpha,k}_{n-j}a^{(k)}_{j}P_{h}(e^{-t_{n}\rho U(x_{0})}G^{0})\zeta^{n}\\
	&\quad+\sum_{n=1}^{\infty}\sum_{i=0}^{n-1}d^{\alpha-1,k}_{i}P_{h}(e^{-t_{i}\rho U(x_{0})}f^{n-i})\zeta^{n}
\\
	&\quad
+\sum_{n=1}^{\infty}\sum_{j=1}^{k-1}d^{\alpha-1,k}_{n-j}a^{(k)}_{j}P_{h}(e^{-t_{n-j}\rho U(x_{0})}f^{0})\zeta^{n}\\
	&\quad+\sum_{l=1}^{k-2}\sum_{j=1}^{k-1}\sum_{n=1}^{\infty}b^{(k)}_{l,j}d^{\alpha-1,k}_{n-j}P_{h}(e^{-t_{n-j}\rho U(x_{0})}\tau^{l}\partial^{l}_{t}f(0))\zeta^{n}.
	\end{aligned}
	\end{equation*}
	Using definitions of $\delta_{\tau,k}$ and doing simple calculations lead to
	\begin{equation*}
	\begin{aligned}
	&((\delta_{\tau,k}(e^{-\tau\rho U(x_{0})}\zeta))^{\alpha}+A_{h})\sum_{n=1}^{\infty}G^{n}_{h}\zeta^{n}\\
	=&P_{h}\left ((\delta_{\tau,k}(e^{-\tau\rho U(x_{0})}\zeta))^{\alpha}\left(\frac{e^{-\tau\rho U(x_{0})}\zeta}{1-e^{-\tau\rho U(x_{0})}\zeta}+\sum_{j=1}^{k-1}a^{(k)}_{j}( e^{-\tau\rho U(x_{0})}\zeta)^{j}\right)G^{0}\right )\\
	&+P_{h}\left ((\delta_{\tau,k}(e^{-\tau\rho U(x_{0})}\zeta))^{\alpha-1}\left(\frac{\zeta}{1-\zeta}+\sum_{j=1}^{k-1}a^{(k)}_{j}\zeta^{j}\right)f^{0}\right )\\
	&+\sum_{l=1}^{k-2}P_{h}\left ((\delta_{\tau,k}(e^{-\tau\rho U(x_{0})}\zeta))^{\alpha-1}\left(\frac{\gamma_{l}(\zeta)}{l!}+\sum_{j=1}^{k-1}b^{(k)}_{l,j}\zeta^{j}\right)\tau^{l}\partial^{l}_{t}f(0)\right )\\
	&+P_{h}\left ((\delta_{\tau,k}(e^{-\tau\rho U(x_{0})}\zeta))^{\alpha-1}\sum_{n=1}^{\infty}R^{n}_{k}\zeta^{n}\right ).\\
	\end{aligned}
	\end{equation*}
	According to Cauchy's integral formula, it holds that
	\begin{equation*}
	\begin{aligned}
	G^{n}_{h}=&\frac{1}{2\pi \mathbf{i}}\int_{|\zeta|=\xi_{\tau}}\zeta^{-n-1}((\delta_{\tau,k}(e^{-\tau\rho U(x_{0})}\zeta))^{\alpha}+A_{h})^{-1}\\
	&\cdot P_{h}\left ((\delta_{\tau,k}(e^{-\tau\rho U(x_{0})}\zeta))^{\alpha}\left(\frac{e^{-\tau\rho U(x_{0})}\zeta}{1-e^{-\tau\rho U(x_{0})}\zeta}+\sum_{n=1}^{k-1}a^{(k)}_{n}(e^{-\tau\rho U(x_{0})}\zeta)^{n}\right)G^{0}\right )d\zeta\\
	&+\frac{1}{2\pi \mathbf{i}}\int_{|\zeta|=\xi_{\tau}}\zeta^{-n-1}((\delta_{\tau,k}(e^{-\tau\rho U(x_{0})}\zeta))^{\alpha}+A_{h})^{-1}\\
	&\cdot P_{h}\left ((\delta_{\tau,k}(e^{-\tau\rho U(x_{0})}\zeta))^{\alpha-1}\left(\frac{\zeta}{1-\zeta}+\sum_{n=1}^{k-1}a^{(k)}_{n}\zeta^{n}\right)f^{0}\right )d\zeta\\
	&+\frac{1}{2\pi \mathbf{i}}\int_{|\zeta|=\xi_{\tau}}\zeta^{-n-1}((\delta_{\tau,k}(e^{-\tau\rho U(x_{0})}\zeta))^{\alpha}+A_{h})^{-1}\\
	&\cdot \sum_{l=1}^{k-2}P_{h}\left ((\delta_{\tau,k}(e^{-\tau\rho U(x_{0})}\zeta))^{\alpha-1}\left(\frac{\gamma_{l}(\zeta)}{l!}+\sum_{j=1}^{k-1}b^{(k)}_{l,j}\zeta^{j}\right)\tau^{l}\partial^{l}_{t}f(0)\right )d\zeta\\
	&+\frac{1}{2\pi \mathbf{i}}\int_{|\zeta|=\xi_{\tau}}\zeta^{-n-1}((\delta_{\tau,k}(e^{-\tau\rho U(x_{0})}\zeta))^{\alpha}+A_{h})^{-1}
	\\
	&\cdot
	P_{h}\left ((\delta_{\tau,k}(e^{-\tau\rho U(x_{0})}\zeta))^{\alpha-1}\sum_{n=1}^{\infty}R^{n}_{k}\zeta^{n}\right )d\zeta,\\
	\end{aligned}
	\end{equation*}
	where $\xi_\tau=e^{-\tau(\kappa+1)}$. Taking $\zeta=e^{-z\tau}$ and deforming $\Gamma^\tau=\{z=\kappa+1+\mathbf{i}y:y\in\mathbb{R}~{\rm and}~|y|\leq \pi/\tau\}$ to
	$\Gamma^\tau_{\theta,\kappa}$ imply the desired results.
\end{proof}

To construct the modification criteria, using \eqref{eqtaylorf}, we rewrite the solution of Eq. \eqref{eqretosol} as
\begin{equation}\label{eqsolresp1}
\begin{aligned}
\tilde{G}=&((\beta(z))^{\alpha}+A)^{-1}((\beta(z))^{\alpha-1}G_{0})\\
&+\sum_{i=0}^{k-2}((\beta(z))^{\alpha}+A)^{-1}((\beta(z))^{\alpha-1}z^{-i-1}\partial^{i}_{t}f(0))\\
&+((\beta(z))^{\alpha}+A)^{-1}((\beta(z))^{\alpha-1}\tilde{R}_{k}).
\end{aligned}
\end{equation}
By comparing \eqref{eqfulldissol} and \eqref{eqsolresp1}, to guarantee $\mathcal{O}(\tau^{k})$ in time and $\mathcal{O}(h^{2})$ in space accuracies at the same time, 
the following estimates are expected, i.e.,
for $z\in \Gamma_{\theta,\kappa}^{\tau}$,
\begin{equation}\label{equreqest}
\begin{aligned}
&\|\mu_{k}(e^{-z\tau})-1\|\leq C|z|^{k}\tau^{k},\\
&\left \|\eta_{k,l}(e^{-z\tau})-\frac{1}{z^{l+1}}\right \|\leq C|z|^{k-l-1}\tau^{k},\quad l=1,2,\ldots ,k-2,\\
&\|\beta_{\tau,k}(z)-\beta(z)\|\leq C|z|^{k+1}\tau^{k},\\
&\|((\beta_{\tau,k}(z))^{\alpha}+A)^{-1}-((\beta_{\tau,k}(z))^{\alpha}+A_h)^{-1} P_h\|\leq Ch^{2}.
\end{aligned}
\end{equation}
It's easy to check that the last two estimates hold automatically; see Lemmas \ref{lemmabetatauOk} and \ref{lemeroper}. For the first two estimates, similar to the derivations of coefficients in Section 2.2 of  \cite{jin2017}, the appropriate choices of  $a^{(k)}_{j}$ and $b^{(k)}_{l,j}$ in \eqref{equreqest} make them hold;  see Tables \ref{tab:defank} and \ref{tabblnk}.

\renewcommand\arraystretch{1.2}
\begin{table}[htbp]
	\caption{Value of $a^{(k)}_{j}$}
	\begin{center}
	\begin{tabular}{|c|ccccc|}
		\hline
		Order & $a^{(k)}_{1}$ & $a^{(k)}_{2}$ & $a^{(k)}_{3}$ & $a^{(k)}_{4}$&$a^{(k)}_{5}$ \\
		\hline
		$k=2$ & $\frac{1}{2}$ &  &  & & \\
		\hline
		$k=3$ & $\frac{11}{12}$ & $-\frac{5}{12}$ &  & &  \\
		\hline
		$k=4$ & $\frac{31}{24}$ & $-\frac{7}{6}$ & $\frac{3}{8}$ &  &\\
		\hline
		$k=5$ & $\frac{1181}{720}$ & $-\frac{177}{80}$ & $\frac{341}{240}$ & $-\frac{251}{270}$ &\\
		\hline
		$k=6$ & $\frac{2837}{1440}$ & $-\frac{2543}{720}$ & $\frac{17}{5}$ & $-\frac{1201}{720}$&$\frac{95}{288}$ \\
		\hline
	\end{tabular}
\end{center}
	\label{tab:defank}
\end{table}

\begin{table}[htbp]
	\caption{Value of $b^{(k)}_{l,j}$}
	\begin{center}
	\begin{tabular}{|c|c|ccccc|}
		\hline
		Order &  & $b^{(k)}_{l,1}$ & $b^{(k)}_{l,2}$ & $b^{(k)}_{l,3}$ & $b^{(k)}_{l,4}$ & $b^{(k)}_{l,5}$ \\
		\hline
		$k=3$ & $l=1$ & $\frac{1}{12}$ & $0$ &  &  &  \\
		\hline
		$k=4$ & $l=1$ & $\frac{1}{6}$ & $-\frac{1}{12}$ & $0$ &  &  \\
		& $l=2$ & $0$ & $0$ & $0$ &  &  \\
		\hline
		$k=5$ & $l=1$ & $\frac{59}{240}$ & $-\frac{29}{120}$ & $\frac{19}{240}$ & $0$ &  \\
		& $l=2$ & $\frac{1}{240}$ & $-\frac{1}{240}$ & $0$ & $0$ &  \\
		& $l=3$ & $-\frac{1}{720}$ & $0$ & $0$ & $0$ &  \\
		\hline
		$k=6$ & $l=1$ & $\frac{77}{240}$ & $-\frac{7}{15}$ & $\frac{73}{240}$ & $-\frac{3}{40}$ & $0$ \\
		& $l=2$ & $\frac{1}{96}$ & $-\frac{1}{60}$ & $\frac{1}{160}$ & $0$ & $0$ \\
		& $l=3$ & $-\frac{1}{360}$ & $\frac{1}{720}$ & $0$ & $0$ & $0$ \\
		& $l=4$ & $0$ & $0$ & $0$ & $0$ & $0$ \\
		\hline
	\end{tabular}
\end{center}
	\label{tabblnk}
\end{table}

\renewcommand\arraystretch{1}

\begin{remark}
	Similar to the proof of Theorem \ref{thmfulldis},
	the solution of \eqref{equncorrect} with $f=0$ can be represented by
	\begin{equation}\label{equncorsol}
	\begin{aligned}
	G^{n}_{h}=&\frac{1}{2\pi \mathbf{i}}\int_{\Gamma^{\tau}_{\theta,\kappa}}e^{zt_{n}}((\beta_{\tau,k}(z))^{\alpha}+A_{h})^{-1}P_{h}\left ((\beta_{\tau,k}(z))^{\alpha-1}e^{-\beta(z)\tau}G^{0}\right )dz.
	\end{aligned}
	\end{equation}
	Motivated by the idea provided in \cite{jin2017}, a general way is to try to rewrite \eqref{equncorsol} as
	 \begin{equation}\label{equncorsolneq}
	 \begin{aligned}
	 ((\beta_{\tau,k}(z))^{\alpha}+A_{h})\sum_{n=1}^{\infty}(G^{n}_{h}-P_{h}e^{-\rho U(x_0)t_n}G^{0})e^{-zt_{n}}=&A_{h}P_{h}\left ((\beta_{\tau,k}(z))^{-1}e^{-\beta(z)\tau}G^{0}\right )
	 \end{aligned}
	 \end{equation}
	 and get $k$-th order scheme by adding some suitable terms to make
	 \begin{equation*}
	 	\left \|\beta_{\tau,k}(z)\left (\frac{e^{-\beta(z)\tau}}{1-e^{-\beta(z)\tau}}+\sum_{j=1}^{k-1}a^{(k)}_{j}e^{-\beta(z)t_{j}}\right )-1\right \|\leq C|z|^{k+1}\tau^{k}.
	 \end{equation*}
	  Thus the correction scheme can be got by using Cauchy's integral formula, which only modifies the $k-1$ starting steps. But \eqref{equncorsolneq} holds only when $U(x_{0})$ is a constant. Here, our modified scheme \eqref{eqfulldissol} can be constructed by modifying
	  \begin{equation}\label{eqcorsol}
	  	((\beta_{\tau,k}(z))^{\alpha}+A_{h})\sum_{n=1}^{\infty}G^{n}_{h}=P_{h}((\beta_{\tau,k}(z))^{\alpha-1}e^{-z\tau}G^{0}).
	  \end{equation}
	  Caused by the term $(\beta_{\tau,k}(z))^{\alpha-1}$ in the right hand of Eq. \eqref{eqcorsol}, we need to modify numerical scheme in each step to keep $O(\tau^k)$ convergence in time.
\end{remark}
\section{Error estimates}
In this section, we first provide the temporal error estimates for the modified high-order BDF scheme. Then the optimal spatial convergence is obtained in $L^{2}$- and $H^{1}$-norm.

Consider the time semi-discrete scheme
\begin{equation}\label{eqsemisch}
\left\{
	\begin{aligned}
	&\sum_{i=0}^{n-1}d^{\alpha,k}_{i}(e^{-t_{i}\rho U(x_{0})}G^{n-i}-(e^{-t_{n}\rho U(x_{0})}G^{0}))+AG^{n}\\
	&\qquad=\sum_{j=1}^{k-1}d^{\alpha,k}_{n-j}a^{(k)}_{j}(e^{-t_{n}\rho U(x_{0})}G^{0})
+\sum_{i=0}^{n-1}d^{\alpha-1,k}_{i}(e^{-t_{i}\rho U(x_{0})}f^{n-i})\\
	&\qquad\quad +\sum_{j=1}^{k-1}a^{(k)}_{j}d^{\alpha-1,k}_{n-j}(e^{-t_{n-j}\rho U(x_{0})}f^{0})
 \\
&\qquad\quad +\sum_{l=1}^{k-2}\sum_{j=1}^{k-1}b^{(k)}_{l,j}\tau^{l}d^{\alpha-1,k}_{n-j}(e^{-t_{n-j}\rho U(x_{0})}\partial^{l}_{t}f(0)),\quad in~ \Omega,\\
	&G^{0}=G_0(x_0),\qquad in ~\Omega,\\
	&G^{0}=0, \qquad on~ \partial\Omega.
	\end{aligned}\right.
\end{equation}
Using the same procedure as in the proof of Theorem \ref{thmfulldis},  the solution of Eq. \eqref{eqsemisch} can be expressed as
\begin{equation}\label{eqtimesemidissol}
\begin{aligned}
G^{n}=&\frac{1}{2\pi \mathbf{i}}\int_{\Gamma^{\tau}_{\theta,\kappa}}e^{zt_{n}}((\beta_{\tau,k}(z))^{\alpha}+A)^{-1}\left ((\beta_{\tau,k}(z))^{\alpha-1}\mu_{k}(e^{-\beta(z)\tau})G^{0}\right )dz\\
&+\frac{1}{2\pi \mathbf{i}}\int_{\Gamma^{\tau}_{\theta,\kappa}}e^{zt_{n}}((\beta_{\tau,k}(z))^{\alpha}+A)^{-1}\left ((\beta_{\tau,k}(z))^{\alpha-1}(\delta_{\tau,k}(e^{-z\tau}))^{-1}\mu_{k}(e^{-z\tau})f^{0}\right )dz\\
&+\frac{1}{2\pi \mathbf{i}}\int_{\Gamma^{\tau}_{\theta,\kappa}}e^{zt_{n}}((\beta_{\tau,k}(z))^{\alpha}+A)^{-1}\sum_{l=1}^{k-2}\left ((\beta_{\tau,k}(z))^{\alpha-1}\eta_{k,l}(e^{-z\tau})\partial^{l}_{t}f(0)\right )dz\\
&+\frac{1}{2\pi \mathbf{i}}\int_{\Gamma^{\tau}_{\theta,\kappa}}e^{zt_{n}}((\beta_{\tau,k}(z))^{\alpha}+A)^{-1}\left ((\beta_{\tau,k}(z))^{\alpha-1}\tau\sum_{n=1}^{\infty}R^{n}_{k}e^{-zt_{n}}\right )dz.
\end{aligned}
\end{equation}
Next, we give two lemmas about $\beta_{\tau,k}(z)$, $k=1,\ldots,6$.
\begin{lemma}\label{lemmabetatauOk}
	Let $\beta_{\tau,k}(z)$ be defined in \eqref{betatauOk} and $U(x_0)$ bounded in $\bar{\Omega}$.  		Denote  $\Sigma^{\tau}_{\theta,\kappa}=\{z\in\mathbb{C}:|z|\geq\kappa,|\arg z|\leq \theta, |Im(z)|\leq \frac{\pi}{\tau},Re(z)\leq \kappa+1\}$ with $Im(z)$ being the imaginary part of $z$ and $Re(z)$ the real part of $z$.   By choosing $\theta\in(\frac{\pi}{2},\pi)$ sufficiently close to $\frac{\pi}{2}$ and $\kappa>0$ large enough $($depending on $|\rho|\|U(x_0)\|_{L^{\infty}(\bar{\Omega})}$$)$, there exists a positive constant $\tau_{*}$ $($depending on $\theta$ and $\kappa$$)$ such that the following estimates hold when $\tau\leq \tau_{*}$:
	\begin{enumerate}[(1)]
		\item For all $x_{0}\in \bar{\Omega}$ and ${{z}}\in \Sigma^{\tau}_{\theta,\kappa}$, we have $\beta_{\tau,k}({{z}}) \in \Sigma_{\pi-\vartheta_{k}+\epsilon,C\kappa}$, where $\vartheta_{k}$ is given in Lemma \ref{lemdk}, and
		\begin{equation*}
		C_1|{{z}}|\leq|\beta_{\tau,k}({{z}})|\leq C_2|{{z}}|.
		\end{equation*}
		
		\item The operator $((\beta_{\tau,k}({{z}}))^\alpha+A)^{-1}:L^2(\Omega)\rightarrow L^2(\Omega)$ is well-defined, bounded, and analytic with respect to $z\in \Sigma^{\tau}_{\theta,\kappa}$, satisfying
		\begin{equation*}
		\|A((\beta_{\tau,k}(z))^\alpha+A)^{-1}\|\leq C~~~~~\forall {{z}} \in \Sigma^{\tau}_{\theta,\kappa},
		\end{equation*}
		\begin{equation*}
		\|((\beta_{\tau,k}({{z}}))^\alpha+A)^{-1}\|\leq C|{{z}}|^{-\alpha}~~~~~\forall{{z}} \in \Sigma^{\tau}_{\theta,\kappa}.
		\end{equation*}
    \item For all $x_{0}\in \bar{\Omega}$ and the real number $\gamma$, there holds
		\begin{equation*}
		|(\beta(z))^{\gamma}-(\beta_{\tau,k}(z))^{\gamma}|\leq C\tau^{k}|z|^{\gamma+k}\quad \forall z\in\Gamma^{\tau}_{\theta,\kappa},
		\end{equation*}
		where $\beta(z)$ is defined in \eqref{defbeta}.
	\end{enumerate}
\end{lemma}

\begin{proof}
	According to \cite{creedon:1975}, we have $\frac{\delta_{k}(\zeta)}{1-\zeta}\neq 0$ in the neighborhood of the unit circle. By choosing $\kappa\geq 8|\rho|\|U(x_{0})\|_{L^{\infty}(\bar{\Omega})}$, $\theta$ sufficiently close to $\frac{\pi}{2}$ and $\tau\leq \frac{\pi}{\kappa+1}$, we have that    $e^{-\tau\beta(z)}$ lies in the neighborhood of the unit circle and there exist two positive constants $C_1$ and $C_2$ such that
	\begin{equation}\label{eqboundz}
		C_{1}\leq \left |\frac{\delta_{k}(e^{-\tau\beta(z)})}{1-e^{-\tau\beta(z)}}\right |\leq C_{2} \quad\forall z\in  \Sigma^{\tau}_{\theta,\kappa}.
	\end{equation}
	From \cite{DengLiQianWang:18}, there holds
	\begin{equation*}
		C_{1}|z|\leq |\delta_{\tau,1}(e^{-\tau\beta(z)})|\leq C_{2}|z|,
	\end{equation*}
	which leads to
	\begin{equation*}
	C_{1}|z|\leq |\delta_{\tau,k}(e^{-\tau\beta(z)})|\leq C_{2}|z|.
	\end{equation*}
Combining  Lemma \ref{lemO200}, one has
	\begin{equation*}
	\beta_{\tau,k}(z)\in\Sigma_{\pi-\vartheta_{k}+\epsilon},
	\end{equation*}
	which yields the second conclusion by using $|\beta_{\tau,k}(z)|\geq C|z|$ and the resolvent estimate \cite{Jin2016}.
	
	As for the third conclusion, there holds
	\begin{equation*}
	\begin{aligned}
	&|(\beta(z))^{\gamma}-(\beta_{\tau,k}(z))^{\gamma}|\\
	=&\left |(\beta(z))^{\gamma}-\left(\beta(z)+\mathcal{O}(\tau^k(\beta(z))^{k+1})\right )^{\gamma}\right |\\
	=&|(\beta(z))^\gamma|\left |1-\left (1+\mathcal{O}(\tau^k(\beta(z))^{k})\right )^{\gamma}\right |.
	\end{aligned}
	\end{equation*}
	If $\tau|\beta(z)|\leq1/2$, we obtain
	\begin{equation*}
	|(\beta(z))^{\gamma}-(\beta_{\tau,k}(z))^{\gamma}|\leq|\beta(z)|^{\gamma}C\tau^k|\beta(z)|^{k}= C\tau^{k}|\beta(z)|^{\gamma+k}.
	\end{equation*}
	As for $\tau|\beta(z)|>1/2$, we have
	\begin{equation*}
	\begin{aligned}
	&\tau|z|\geq C\tau|\beta_{\tau,k}(z)|\geq C \qquad \forall z\in\Gamma^{\tau}_{\theta,\kappa},\\
	&|(\beta(z))^{\gamma}-(\beta_{\tau,k}(z))^{\gamma}|\leq C|z|^{\gamma}\leq C\tau^{k}|z|^{\gamma+k} \qquad \forall z\in\Gamma^{\tau}_{\theta,\kappa}.\\
	\end{aligned}
	\end{equation*}
	Thus the third conclusion is obtained.
\end{proof}

\begin{lemma}\label{lemO200}
	Let $\delta_{\tau,k}(e^{-z\tau})$ be defined in \eqref{betatauOk}, $U(x_{0})$ bounded in $\bar{\Omega}$,  and $L=|{{\rho}}|\|U({{x_{0}}})\|_{L^{\infty}(\bar{\Omega})}$. There exist positive constants $\theta_{0}\in\left (\frac{\pi}{2},\frac{9\pi}{16}\right )$ and $\tau_0$ such that if $\theta\in\left(\frac{\pi}{2},\theta_0\right )$ and $\tau\in (0,\tau_0]$, then
	\begin{equation}\label{equlem00O2con}
	\begin{aligned}
	\delta_{\tau,k}(e^{-\beta(z)\tau})\in \Sigma_{\pi-\vartheta_{k}+\epsilon} \quad  ~~\forall z\in \Gamma^{\tau}_{\theta,\kappa}{~~\rm and~~}\forall x_{0}\in \bar{\Omega}.
	\end{aligned}
	\end{equation}
\end{lemma}
\begin{proof}
Take $\kappa \, (>8L)$ sufficiently large, $r=Re(\rho U(x_{0}))$, and $\omega=\mathbf{i} \cdot Im(\rho U(x_{0}))$ for $x_{0}\in \bar{\Omega}$. Here we choose $\tau\leq \frac{\pi}{\kappa+1}$ to make $e^{-\tau r}$ lie in the neighborhood of $1$.
Taylor's expansion and $\tau|z|\leq C$ give
\begin{equation}\label{eqneg1}
	\begin{aligned}
		\left |\frac{z\tau e^{-z\tau}}{1- e^{-z\tau}}\right |\leq 1+O(|z|\tau)\leq C.
	\end{aligned}
\end{equation}
Combining \eqref{eqboundz}, \eqref{eqneg1}, and the bound of $|\tau\delta_{\tau,k}'(e^{-z\tau})|$, i.e., $|\tau\delta_{\tau,k}'(e^{-z\tau})|\leq 1+\tau|\delta_{\tau,k-1}(e^{-z\tau})|\leq C$ for $k>1$ and $|\tau\delta_{\tau,1}'(e^{-z\tau})|\leq C$, we have, for some $\sigma\in(0,1)$
\begin{equation*}
	\begin{aligned}
		 \frac{|\delta_{\tau,k}(e^{-r\tau} e^{-z\tau})-\delta_{\tau,k}(e^{-z\tau})|}{|\delta_{\tau,k}( e^{-z\tau})|}
		& \leq C\left | \frac{\delta_{\tau,k}(e^{-r\tau} e^{-z\tau})}{\delta_{\tau,k}( e^{-z\tau})}-1\right |\\
		&\leq  CL\left | \frac{\delta_{\tau,k}'(e^{-\sigma r\tau} e^{-z\tau})e^{-z\tau}e^{-\sigma r\tau}\tau}{\delta_{\tau,k}( e^{-z\tau})}\right |\\
		&\leq  CL|\tau\delta_{\tau,k}'(e^{-\sigma r\tau} e^{-z\tau})|\left | \frac{e^{-z\tau}}{\delta_{\tau,k}( e^{-z\tau})}\right |\\
		&\leq CL\left | \frac{\tau e^{-z\tau}}{1- e^{-z\tau}}\right |\\
	&	\leq C\frac{L}{\kappa},
	\end{aligned}
\end{equation*}
where $\delta_{\tau,k}'(\zeta)$ is the first order derivative about $\zeta$. Taking $\kappa$ large enough results in
\begin{equation*}
	|\arg(\delta_{\tau,k}(e^{-r\tau} e^{-z\tau}))-\arg(\delta_{\tau,k}(e^{-z\tau}))|\leq \epsilon/4.
\end{equation*}
Similarly,  there holds, for some $\sigma\in(0,1)$
\begin{equation*}
\begin{aligned}
\frac{|\delta_{\tau,k}(e^{-\omega\tau}e^{-r\tau}e^{-z\tau})-\delta_{\tau,k}(e^{-r\tau} e^{-z\tau})|}{|\delta_{\tau,k}( e^{-r\tau}e^{-z\tau})|}\leq& C\left | \frac{\delta_{\tau,k}(e^{-\omega\tau}e^{-r\tau} e^{-z\tau})}{\delta_{\tau,k}(e^{-r\tau} e^{-z\tau})}-1\right |\\
\leq &CL\left | \frac{\delta_{\tau,k}'(e^{-\sigma \omega\tau}e^{-r\tau} e^{-z\tau})e^{-z\tau}e^{-\sigma \omega\tau} e^{-r\tau}\tau}{\delta_{\tau,k}(e^{-r\tau} e^{-z\tau})}\right |\\
\leq &CL|\tau\delta_{\tau,k}'(e^{-\sigma \omega\tau}e^{-r\tau} e^{-z\tau})|\left | \frac{e^{-r\tau}e^{-z\tau}}{\delta_{\tau,k}( e^{-z\tau}e^{-r\tau})}\right |\\
\leq&CL\left | \frac{\tau e^{-z\tau}e^{-r\tau}}{1- e^{-z\tau}e^{-r\tau}}\right |\\
\leq&C\frac{L}{\kappa}.
\end{aligned}
\end{equation*}
Again, when $\kappa$ is large enough, it holds
\begin{equation*}
|\arg(\delta_{\tau,k}(e^{-\omega\tau}e^{-r\tau} e^{-z\tau}))-\arg(\delta_{\tau,k}(e^{-r\tau}e^{-z\tau}))|\leq \epsilon/4.
\end{equation*}
From \cite{jin2017}, we have
\begin{equation*}
	\delta_{\tau,k}(e^{-z\tau})\in \Sigma_{\pi-\vartheta_k+\epsilon/2}.
\end{equation*}
Thus
\begin{equation*}
	\beta_{\tau,k}(z)=\delta_{\tau,k}(e^{-\tau(z+\rho U(x_{0}))})\in \Sigma_{\pi-\vartheta_k+\epsilon}.
\end{equation*}
\end{proof}

According to the above two lemmas, the following temporal error estimates can be obtained.
\begin{theorem}\label{thmsemierrorOk}
	Let $G(t_{n})$ and $G^n$  be the solutions of Eqs. \eqref{eqretosol} and \eqref{eqsemisch}, respectively. Assume $U(x_{0})$ is bounded in $\bar{\Omega}$. If $G_0\in L^2(\Omega)$, $f\in C^{k-1}([0,T],L^{2}(\Omega))$, and $\int_{0}^{t}\|\partial^{k}_{t}f(s)\|_{L^2(\Omega)}ds<\infty$, then there holds
	\begin{equation*}
	\begin{aligned}
		\|G(t_{n})-G^{n}\|_{L^2(\Omega)}\leq& Ct_n^{-k}\tau^{k}\|G_0\|_{L^2(\Omega)}+C\sum_{l=0}^{k-1}\tau^{k}t^{l+1-k}_{n}\|\partial^{l}_{t}f(0)\|_{L^2(\Omega)}\\
		&+C\tau^{k}\int_{0}^{t_{n}}\|\partial^{k}_{t}f(s)\|_{L^2(\Omega)}ds.
	\end{aligned}
	\end{equation*}
\end{theorem}
\begin{proof}
	Subtracting \eqref{eqtimesemidissol} from \eqref{eqsolresp1} leads to
	\begin{equation*}
	\begin{aligned}
	\|G(t_{n})-G^{n}\|_{L^2(\Omega)}	\leq&C(\uppercase\expandafter{\romannumeral1}+\uppercase\expandafter{\romannumeral2}+\uppercase\expandafter{\romannumeral3}+\uppercase\expandafter{\romannumeral4}),
	\end{aligned}
	\end{equation*}
	where
	\begin{equation*}
	\begin{aligned}
	\uppercase\expandafter{\romannumeral1}\leq& C\bigg\|\int_{\Gamma_{\theta,\kappa}}e^{zt_{n}}((\beta(z))^{\alpha}+A)^{-1}((\beta(z))^{\alpha-1}G^{0})dz\\
	&\quad-\int_{\Gamma^{\tau}_{\theta,\kappa}}e^{zt_{n}}((\beta_{\tau,k}(z))^{\alpha}+A)^{-1}\left ((\beta_{\tau,k}(z))^{\alpha-1}\mu_{k}(e^{-\beta(z)\tau})G^{0}\right )dz\bigg\|_{L^2(\Omega)},\\
	\uppercase\expandafter{\romannumeral2}\leq& C\bigg\|\int_{\Gamma_{\theta,\kappa}}e^{zt_{n}}((\beta(z))^{\alpha}+A)^{-1}((\beta(z))^{\alpha-1}z^{-1}f^{0})dz\\
	&\quad-\int_{\Gamma^{\tau}_{\theta,\kappa}}e^{zt_{n}}((\beta_{\tau,k}(z))^{\alpha}+A)^{-1}
\\
	&\quad \cdot
\left  ((\beta_{\tau,k}(z))^{\alpha-1}(\delta_{\tau,k}(e^{-z\tau}))^{-1}\mu_{k}(e^{-z\tau})f^{0}\right )dz\bigg\|_{L^2(\Omega)},\\
	\uppercase\expandafter{\romannumeral3}\leq& C\bigg\|\int_{\Gamma_{\theta,\kappa}}e^{zt_{n}}((\beta(z))^{\alpha}+A)^{-1}\sum_{l=1}^{k-2}\left ((\beta(z))^{\alpha-1}z^{-l-1}\partial^{l}_{t}f(0)\right )dz\\
	&\quad-\int_{\Gamma^{\tau}_{\theta,\kappa}}e^{zt_{n}}((\beta_{\tau,k}(z))^{\alpha}+A)^{-1}\sum_{l=1}^{k-2}\left ((\beta_{\tau,k}(z))^{\alpha-1}\eta_{k,l}(e^{-z\tau})\partial^{l}_{t}f(0)\right )dz\bigg\|_{L^2(\Omega)},\\
	\uppercase\expandafter{\romannumeral4}\leq& C\bigg\|\int_{\Gamma_{\theta,\kappa}}e^{zt_{n}}((\beta(z))^{\alpha}+A)^{-1}\left ((\beta(z))^{\alpha-1}\tilde{R}_{k}\right )dz\\
	&\quad-\int_{\Gamma^{\tau}_{\theta,\kappa}}e^{zt_{n}}((\beta_{\tau,k}(z))^{\alpha}+A)^{-1}\left ((\beta_{\tau,k}(z))^{\alpha-1}\tau\sum_{n=1}^{\infty}R^{n}_{k}e^{-zt_{n}}\right )dz\bigg\|_{L^2(\Omega)}.\\
	\end{aligned}
	\end{equation*}
	For $\uppercase\expandafter{\romannumeral1}$, it has
	\begin{equation*}
	\begin{aligned}
	\uppercase\expandafter{\romannumeral1}\leq& C\bigg\|\int_{\Gamma_{\theta,\kappa}\backslash\Gamma^{\tau}_{\theta,\kappa}}e^{zt_{n}}((\beta(z))^{\alpha}+A)^{-1}((\beta(z))^{\alpha-1}G^{0})dz\bigg\|_{L^2(\Omega)}\\
	&\quad+\bigg\|\int_{\Gamma^{\tau}_{\theta,\kappa}}e^{zt_{n}}\left (((\beta(z))^{\alpha}+A)^{-1}\left ((\beta(z))^{\alpha-1}G_{0}\right)\right.\\
	&\left. \quad -((\beta_{\tau,k}(z))^{\alpha}+A)^{-1}\left ((\beta_{\tau,k}(z))^{\alpha-1}\mu_{k}(e^{-\beta(z)\tau})G^{0}\right )\right )dz\bigg\|_{L^2(\Omega)}.\\
	\end{aligned}
	\end{equation*}
Combining Eq. \eqref{equreqest} and Lemmas \ref{lemmaBeta} and \ref{lemmabetatauOk} yields
	\begin{equation*}
	\begin{aligned}
	&\left\|((\beta(z))^{\alpha}+A)^{-1}\left ((\beta(z))^{\alpha-1}G_{0}\right)-((\beta_{\tau,k}(z))^{\alpha}+A)^{-1} \right.
\\
	&
\cdot\left.\left ((\beta_{\tau,k}(z))^{\alpha-1}\mu_{k}(e^{-\beta(z)\tau})G^{0}\right )\right\|_{L^2(\Omega)}\\
&	\leq \left\|((\beta(z))^{\alpha}+A)^{-1}\left ((\beta(z))^{\alpha-1}G_{0}\right)-((\beta_{\tau,k}(z))^{\alpha}+A)^{-1}\left ((\beta(z))^{\alpha-1}G_{0}\right)\right\|_{L^2(\Omega)}\\
	&+\left\|((\beta_{\tau,k}(z))^{\alpha}+A)^{-1}\left ((\beta(z))^{\alpha-1}G_{0}\right)-((\beta_{\tau,k}(z))^{\alpha}+A)^{-1}\right.
\\
	&
\cdot \left.\left ((\beta_{\tau,k}(z))^{\alpha-1}\mu_{k}(e^{-\beta(z)\tau})G^{0}\right )\right\|_{L^2(\Omega)}\\
&	\leq Cz^{k-1}\tau^{k}\|G_{0}\|_{L^2(\Omega)},
	\end{aligned}
	\end{equation*}
	which leads to
	\begin{equation*}
	\uppercase\expandafter{\romannumeral1}\leq C\tau^{k}t^{-k}_{n}\|G_{0}\|_{L^2(\Omega)}.
	\end{equation*}
	Similarly, we obtain
	\begin{equation*}
	\uppercase\expandafter{\romannumeral2}\leq C\tau^{k}t^{1-k}_{n}\|f^{0}\|_{L^2(\Omega)}.
	\end{equation*}
	As for $\uppercase\expandafter{\romannumeral3}$, we have
	\begin{equation*}
	\begin{aligned}
	\uppercase\expandafter{\romannumeral3}\leq& C\bigg\|\int_{\Gamma_{\theta,\kappa}\backslash \Gamma^{\tau}_{\theta,\kappa}}e^{zt_{n}}((\beta(z))^{\alpha}+A)^{-1}\sum_{l=1}^{k-2}\left ((\beta(z))^{\alpha-1}z^{-l-1}\partial^{l}_{t}f(0)\right )dz\bigg\|_{L^2(\Omega)}\\
	&\quad+\bigg\|\int_{\Gamma^{\tau}_{\theta,\kappa}}e^{zt_{n}}\bigg(((\beta(z))^{\alpha}+A)^{-1}\sum_{l=1}^{k-2}\left ((\beta(z))^{\alpha-1}z^{-l-1}\partial^{l}_{t}f(0)\right )\\
	&-((\beta_{\tau,k}(z))^{\alpha}+A)^{-1}\sum_{l=1}^{k-2}\left ((\beta_{\tau,k}(z))^{\alpha-1}\eta_{k,l}(e^{-z\tau})\partial^{l}_{t}f(0)\right )\bigg)dz\bigg\|_{L^2(\Omega)}.
	\end{aligned}
	\end{equation*}
Combining Eq. \eqref{equreqest} and Lemmas \ref{lemmaBeta} and \ref{lemmabetatauOk} gives
	\begin{equation*}
		\begin{aligned}
		&\bigg\|((\beta(z))^{\alpha}+A)^{-1}\left ((\beta(z))^{\alpha-1}z^{-l-1}\partial^{l}_{t}f(0)\right )\\
		&\quad-((\beta_{\tau,k}(z))^{\alpha}+A)^{-1}\left ((\beta_{\tau,k}(z))^{\alpha-1}\eta_{k,l}(e^{-z\tau})\partial^{l}_{t}f(0)\right )\bigg\|_{L^2(\Omega)}\\
		\leq&\left\|\left(((\beta(z))^{\alpha}+A)^{-1}-((\beta_{\tau,k}(z))^{\alpha}+A)^{-1}\right)\left ((\beta(z))^{\alpha-1}z^{-l-1}\partial^{l}_{t}f(0)\right )\right\|_{L^2(\Omega)}\\
		&+\left\|((\beta_{\tau,k}(z))^{\alpha}+A)^{-1}\left ((\beta(z))^{\alpha-1}z^{-l-1}\partial^{l}_{t}f(0) \right.\right.
\\
		& \left.\left.
-((\beta_{\tau,k}(z))^{\alpha-1}\eta_{k,l}(e^{-z\tau})\partial^{l}_{t}f(0)\right )\right\|_{L^2(\Omega)}\\
		\leq &Cz^{k-l-2}\tau^{k}\|\partial^{l}_{t}f(0)\|_{L^2(\Omega)},
		\end{aligned}
	\end{equation*}
	which implies
	\begin{equation*}
	\uppercase\expandafter{\romannumeral3}\leq C\tau^{k}\sum_{l=1}^{k-2}t^{l+1-k}_{n}\|\partial^{l}_{t}f(0)\|_{L^2(\Omega)}.
	\end{equation*}
	As for $\uppercase\expandafter{\romannumeral4}$, it has
	\begin{equation*}
		\uppercase\expandafter{\romannumeral4}\leq \uppercase\expandafter{\romannumeral4}_{1}+\uppercase\expandafter{\romannumeral4}_{2},
	\end{equation*}
	where
	\begin{equation*}
		\begin{aligned}
		\uppercase\expandafter{\romannumeral4}_{1}\leq& C\bigg\|\int_{\Gamma_{\theta,\kappa}}e^{zt_{n}}((\beta(z))^{\alpha}+A)^{-1}\left ((\beta(z))^{\alpha-1}z^{-k}\partial^{k-1}_{t}f(0)\right )dz\\
		&\quad-\int_{\Gamma^{\tau}_{\theta,\kappa}}e^{zt_{n}}((\beta_{\tau,k}(z))^{\alpha}+A)^{-1}
\\
		&\quad
\cdot\left ((\beta_{\tau,k}(z))^{\alpha-1}\sum_{n=1}^{\infty}\frac{t_{n}^{k-1}}{(k-1)!}\partial^{k-1}_{t}f(0)e^{-zt_{n}}\right )dz\bigg\|_{L^2(\Omega)},\\
		\uppercase\expandafter{\romannumeral4}_{2}\leq& C\bigg\|\frac{1}{2\pi\mathbf{i}}\int_{\Gamma_{\theta,\kappa}}e^{zt_{n}}((\beta(z))^{\alpha}+A)^{-1}\left ((\beta(z))^{\alpha-1}\tilde{f}\right )dz\\
		&\quad-\frac{1}{2\pi\mathbf{i}}\int_{\Gamma^{\tau}_{\theta,\kappa}}e^{zt_{n}}((\beta_{\tau,k}(z))^{\alpha}+A)^{-1}
\\
		&\quad
\cdot\left ((\beta_{\tau,k}(z))^{\alpha-1}\sum_{n=1}^{\infty}\left (\frac{t^{k-1}}{(k-1)!}\ast\partial^{k}_{t}f\right )(t_{n})e^{-zt_{n}}\right )dz\bigg\|_{L^2(\Omega)}.\\
		\end{aligned}
	\end{equation*}
Simple calculations \cite{lubich1996} lead to
	\begin{equation}
		\uppercase\expandafter{\romannumeral4}\leq C\tau^{k}\|\partial^{k-1}_{t}f(0)\|_{L^2(\Omega)}+C\tau^{k}\int_{0}^{t_{n}}\|\partial^{k}_{t}f(s)\|_{L^2(\Omega)}ds.
	\end{equation}
\end{proof}

Now, we provide the spatial error estimate.

\begin{lemma}[\cite{Sun:2020}]\label{lemeroper}
	Let $v\in L^2(\Omega)$, $U(x_{0})$ be bounded in $\bar{\Omega}$ and $z\in\Sigma^{\tau}_{\theta,\kappa}$ with $\kappa$ largely enough, where $\Sigma^{\tau}_{\theta,\kappa}$ is defined in Lemma \ref{lemmabetatauOk}. Denote $w=((\beta_{\tau,k}(z))^{\alpha}+A)^{-1}v$ and $w_h=((\beta_{\tau,k}(z))^{\alpha}+A_h)^{-1} P_hv$, where $\beta_{\tau,k}(z)$ is defined in \eqref{betatauOk}. Then one has
	\begin{equation*}
	\|w-w_h\|_{L^2(\Omega)}+h\|w-w_h\|_{\dot{H}^{1}(\Omega)}\leq Ch^{2}\|v\|_{L^2(\Omega)}.
	\end{equation*}
\end{lemma}
\begin{proof}
	The results can be similarly obtained as the proof in \cite{Sun:2020}.
\end{proof}

\begin{theorem}\label{thmfullerrorO2}
	Let $G^{n}$ and $G^{n}_{h}$  be the solutions of Eqs. \eqref{eqtimesemidissol} and \eqref{eqfullscheme} respectively and assume $G_0\in L^2(\Omega)$, $f\in C^{k}([0,T],L^{2}(\Omega))$, $\int_{0}^{t}\|\partial^{k}_{t}f(s)\|_{L^2(\Omega)}ds<\infty$ and $U(x_{0})$ is bounded in $\bar{\Omega}$. Then we have
	\begin{equation}
	\begin{aligned}
	&\|G^{n}-G^{n}_{h}\|_{L^2(\Omega)}+h\|G^{n}-G^{n}_{h}\|_{\dot{H}^{1}(\Omega)}\\
	&\qquad\qquad \leq Ch^{2}t_{n}^{-\alpha}\|G_0\|_{L^2(\Omega)}+Ch^{2}\sum_{l=0}^{k-1}\|\partial^{l}_{t}f(0)\|_{L^2(\Omega)}+Ch^{2}\int_{0}^{t_{n}}\left \|\partial^{k}_{t}f(s)\right \|_{L^2(\Omega)}ds.
	\end{aligned}
	\end{equation}
\end{theorem}
\begin{proof}
	Subtracting \eqref{eqfulldissol} from \eqref{eqtimesemidissol} leads to
	\begin{equation*}
		\begin{aligned}
		&\|G^{n}-G^{n}_{h}\|_{L^2(\Omega)}\\
		\leq&C\left \|\int_{\Gamma^{\tau}_{\theta,\kappa}}e^{zt_{n}}\left (((\beta_{\tau,k}(z))^{\alpha}+A)^{-1}-((\beta_{\tau,k}(z))^{\alpha}+A_{h})^{-1}P_{h}\right )
 \right.
\\
		&
\cdot \left.
\left ((\beta_{\tau,k}(z))^{\alpha-1}\mu_{k}(e^{-\beta(z)\tau})G^{0}\right )dz\right \|_{L^2(\Omega)}\\
		&+C\left \|\int_{\Gamma^{\tau}_{\theta,\kappa}}e^{zt_{n}}
\left (((\beta_{\tau,k}(z))^{\alpha}+A)^{-1}-((\beta_{\tau,k}(z))^{\alpha}+A_{h})^{-1}P_{h}\right )
\right.
\\
		&
\cdot \left.
\left ((\beta_{\tau,k}(z))^{\alpha-1}(\delta_{\tau,k}(e^{-z\tau}))^{-1}\mu_{k}(e^{-z\tau})f^{0}\right )dz\right \|_{L^2(\Omega)}\\
		&+C\left \|\int_{\Gamma^{\tau}_{\theta,\kappa}}e^{zt_{n}}
\left (((\beta_{\tau,k}(z))^{\alpha}+A)^{-1}-((\beta_{\tau,k}(z))^{\alpha}+A_{h})^{-1}P_{h}\right )
\right.
\\
		&
\cdot \left.
\sum_{l=1}^{k-2}\left ((\beta_{\tau,k}(z))^{\alpha-1}\eta_{k,l}(e^{-z\tau})\partial^{l}_{t}f(0)\right )dz\right \|_{L^2(\Omega)}\\
		&+C\left \|\int_{\Gamma^{\tau}_{\theta,\kappa}}e^{zt_{n}}\left (((\beta_{\tau,k}(z))^{\alpha}+A)^{-1}-((\beta_{\tau,k}(z))^{\alpha}+A_{h})^{-1}P_{h}\right )
\right.
\\
		&
\cdot \left.
\left ((\beta_{\tau,k}(z))^{\alpha-1}\tau\sum_{n=1}^{\infty}R^{n}_{k}e^{-zt_{n}}\right )dz\right \|_{L^2(\Omega)}.
		\end{aligned}
	\end{equation*}
	Using Lemmas \ref{lemmabetatauOk} and \ref{lemeroper} lead to
	\begin{equation*}
		\|G^{n}-G^{n}_{h}\|_{L^2(\Omega)}\leq Ch^{2}t_{n}^{-\alpha}\|G_{0}\|_{L^2(\Omega)}+Ch^{2}\sum_{l=0}^{k-1}\|\partial^{l}_{t}f(0)\|_{L^2(\Omega)}+Ch^{2}\int_{0}^{t_{n}}\left \|\partial^{k}_{t}f(s)\right \|_{L^2(\Omega)}ds.
	\end{equation*}
	Similarly, we have
	\begin{equation*}
	\|G^{n}-G^{n}_{h}\|_{\dot{H}^1(\Omega)}\leq Cht_{n}^{-\alpha}\|G_{0}\|_{L^2(\Omega)}+Ch\sum_{l=0}^{k-1}\|\partial^{l}_{t}f(0)\|_{L^2(\Omega)}+Ch\int_{0}^{t_{n}}\left \|\partial^{k}_{t}f(s)\right \|_{L^2(\Omega)}ds.
	\end{equation*}
\end{proof}

Lastly, according to Theorems \ref{thmsemierrorOk} and \ref{thmfullerrorO2}, the following error estimates of the fully discrete scheme are reached.
\begin{theorem}\label{thmfull}
	Let $G(t)$ and $G_{h}^{n}$  be the solutions of Eqs. \eqref{eqretosol} and \eqref{eqfullscheme} respectively and assume $G_0\in L^2(\Omega)$,  $f\in C^{k-1}([0,T],L^{2}(\Omega))$, $\int_{0}^{t}(t-s)^{-\alpha}\|f(s)\|_{L^2(\Omega)}ds<\infty$, $\int_{0}^{t}\|\partial^{k}_{t}f(s)\|_{L^2(\Omega)}ds<\infty$, and $U(x_{0})$ is bounded in $\bar{\Omega}$. Then one has
	\begin{equation}
	\begin{aligned}
	\|G(t_{n})-G_{h}^{n}\|_{L^2(\Omega)}\leq& Ch^{2}t_{n}^{-\alpha}\|G_{0}\|_{L^2(\Omega)}+Ch^{2}\sum_{l=0}^{k-1}\|\partial^{l}_{t}f(0)\|_{L^2(\Omega)}\\
&+Ch^{2}\int_{0}^{t_{n}}\left \|\partial^{k}_{t}f(s)\right \|_{L^2(\Omega)}ds\\
	&+Ct_n^{-k}\tau^{k}\|G_0\|_{L^2(\Omega)}+C\tau^{k}\sum_{l=0}^{k-1}t_{n}^{l+1-k}\|\partial^{l}_{t}f_{0}\|_{L^2(\Omega)}\\
	&+C\tau^{k}\int_{0}^{t_{n}}\|\partial^{k}_{t}f(s)\|_{L^2(\Omega)}ds.
	\end{aligned}
	\end{equation}
\end{theorem}
\section{Numerical experiments}
In this section, we perform the numerical experiments to verify the effectiveness of the designed schemes. Since the exact solution $G$ is unknown, to test the space convergence rates, we denote by 
\begin{equation*}
\begin{aligned}
E_{h}=\|G^{n}_{h}-G^{n}_{h/2}\|_{L^2(\Omega)},
\end{aligned}
\end{equation*}
where $G^n_{h}$ means the numerical solution of $G$ at time $t_n$ with mesh size $h$; similarly, to test the temporal convergence rates, we take
\begin{equation*}
\begin{aligned}
E_{\tau}=\|G_{\tau}-G_{\tau/2}\|_{L^2(\Omega)},
\end{aligned}
\end{equation*}
where $G_{\tau}$ is the numerical solution of $G$ at the fixed time $t$ with step size $\tau$. The spatial and temporal convergence rates can be, respectively, obtained by calculating
\begin{equation*}
{\rm Rate}=\frac{\ln(E_{h}/E_{h/2})}{\ln(2)},\quad {\rm Rate}=\frac{\ln(E_{\tau}/E_{\tau/2})}{\ln(2)}.
\end{equation*}
For convience, we take $\Omega=(0,1)$ and choose $T=1$ as a terminal time in the following examples. And all the computations are carried out in Julia 1.4.3 on a personal laptop. To observe the temporal convergence rates clearly, we use 80-bit precision float to do the computation and save the data.
\begin{example}
In this example, we take $\rho=-1$,
\begin{equation*}
	f(x_0,\rho,t)=0,\quad G(x_{0},\rho,0)=x_{0}(1-x_{0})~~{\rm and}~~U(x_{0})=\chi_{(0.5,1)}(x_{0}),
\end{equation*}	
where $\chi_{(a,b)}$ denotes the characteristic function on $(a,b)$. To investigate the convergence in temporal direction and
eliminate the influence from spatial discretization, we take $h=1/100$. The corresponding temporal errors and convergence rates are presented in Table \ref{Tab:homtime} for scheme \eqref{eqfullscheme}. All the convergence rates are steady and can reach up to order $6$.
\begin{table}[htbp]
	\caption{Temporal errors and convergence rates}
	\begin{tabular}{|c|c|ccccc|c|}
		\hline
		$\alpha$ & $k\backslash 1/\tau$ & 50 & 100 & 200 & 400 & 800 & Rate \\
		\hline
	& 2 & 1.3916E-06 & 3.4220E-07 & 8.4846E-08 & 2.1124E-08 & 5.2703E-09&$\approx2.0029$ \\
	& 3 & 6.6959E-08 & 8.0530E-09 & 9.8763E-10 & 1.2229E-10 & 1.5214E-11&$\approx3.0068$ \\
	0.3 & 4 & 4.5036E-09 & 2.6218E-10 & 1.5824E-11 & 9.7204E-13 & 6.0232E-14&$\approx4.0124$ \\
	& 5 & 4.1599E-10 & 1.1158E-11 & 3.2977E-13 & 1.0025E-14 & 3.0900E-16&$\approx5.0197$ \\
	& 6 & 3.6025E-07 & 4.3156E-11 & 8.5287E-15 & 1.2764E-16 & 1.9547E-18&$\approx6.0290$ \\
	\hline
	& 2 & 3.6346E-06 & 8.8919E-07 & 2.1988E-07 & 5.4670E-08 & 1.3630E-08&$\approx2.0039$ \\
	& 3 & 2.1375E-07 & 2.5479E-08 & 3.1112E-09 & 3.8441E-10 & 4.7774E-11&$\approx3.0083$ \\
	0.7 & 4 & 1.6696E-08 & 9.6070E-10 & 5.7657E-11 & 3.5318E-12 & 2.1854E-13&$\approx4.0144$ \\
	& 5 & 1.8665E-09 & 4.5724E-11 & 1.3418E-12 & 4.0650E-14 & 1.2509E-15&$\approx5.0222$ \\
	& 6 & 2.7567E-06 & 3.1378E-09 & 3.7837E-14 & 5.6558E-16 & 8.6440E-18&$\approx6.0318$ \\
		\hline
	\end{tabular}
	\label{Tab:homtime}
\end{table}
	
\end{example}

\begin{example}
	Here, we take $\rho=-1$,
	\begin{equation*}
	f(x_0,\rho,t)=x_{0}(1-x_{0})e^{-\rho\chi_{(0.5,1)}(x_{0})t},\quad G(x_{0},\rho,0)=0~~{\rm and}~~U(x_{0})=\chi_{(0.5,1)}(x_{0}).
	\end{equation*}	
	 To avoid the influence on temporal errors from the spatial discretization, we choose $h=1/100$. We use \eqref{eqfullscheme} to solve \eqref{eqretosol} and present the corresponding temporal errors and convergence rates in Table \ref{Tab:Nontime}.  The convergence rates are steady and can reach up to order $6$.
	\begin{table}[htbp]
		\caption{Temporal errors and convergence rates}
		\begin{tabular}{|c|c|ccccc|c|}
			\hline
			$\alpha$ & $k\backslash 1/\tau$ & 50 & 100 & 200 & 400 & 800 & Rate \\
			\hline
			& 2 & 1.8340E-06 & 4.5996E-07 & 1.1517E-07 & 2.8817E-08 & 7.2072E-09 & $\approx$ 1.9994 \\
			& 3 & 3.0701E-08 & 3.8758E-09 & 4.8686E-10 & 6.1006E-11 & 7.6351E-12 & $\approx$ 2.9982 \\
			0.4 & 4 & 4.5287E-10 & 2.8923E-11 & 1.8274E-12 & 1.1484E-13 & 7.1968E-15 & $\approx$ 3.9961 \\
			& 5 & 1.9285E-11 & 7.6798E-13 & 2.3449E-14 & 7.2459E-16 & 2.2518E-17 & $\approx$ 5.0080 \\
			& 6 & 9.3685E-08 & 2.6523E-11 & 3.5161E-16 & 5.4074E-18 & 8.2994E-20 & $\approx$ 6.0258 \\
			\hline
			& 2 & 7.6913E-07 & 1.9366E-07 & 4.8588E-08 & 1.2169E-08 & 3.0449E-09 & $\approx$ 1.9987 \\
			& 3 & 2.5894E-08 & 3.2146E-09 & 4.0050E-10 & 4.9982E-11 & 6.2428E-12 & $\approx$ 3.0011 \\
			0.6 & 4 & 4.7283E-10 & 2.6989E-11 & 1.6111E-12 & 9.8392E-14 & 6.0786E-15 & $\approx$ 4.0167 \\
			& 5 & 6.1981E-11 & 1.7996E-12 & 5.3852E-14 & 1.6473E-15 & 5.0933E-17 & $\approx$ 5.0153 \\
			& 6 & 6.2004E-08 & 6.2135E-11 & 1.1952E-15 & 1.8010E-17 & 2.7633E-19 & $\approx$ 6.0262 \\
			\hline
		\end{tabular}
		\label{Tab:Nontime}
	\end{table}
\end{example}

\begin{example}
In this example, we take $\rho=-1+\pi\mathbf{i}$,
\begin{equation*}
\begin{aligned}
&f(x_0,\rho,t)=0,\quad G(x_{0},\rho,0)=-5 \chi_{(0,0.5)}(x_{0})+5 \chi_{(0.5,1)}(x_{0}),\\
&\qquad{\rm and}~~U(x_{0})=3 (x_{0}+0.5)^5 \chi_{(0,0.5)}(x_{0}).
\end{aligned}
\end{equation*}	
We choose $\tau=1/200$ to decrease the errors caused by temporal discretizations.
We use \eqref{eqfullscheme} to solve Eq. \eqref{eqretosol} and present the $L^{2}$- and $H^{1}$-norm errors and convergence rates in Tables \ref{Tab:homspaceL2} and \ref{Tab:homspaceH1}, respectively. All the convergence rates are consistent with the predicted results.

	\begin{table}[htbp]
		\caption{$L^{2}$-norm errors in space and convergence rates}
		\begin{tabular}{|c|c|ccccc|c|}
			\hline
			$\alpha$ & $k\backslash 1/h$ & 20 & 40 & 80 & 160 & 4096 & Rate \\
			\hline
			& 2 & 7.0515E-06 & 1.7618E-06 & 4.4038E-07 & 1.1012E-07 & 2.7523E-08 & $\approx$ 2.0003 \\
			& 3 & 7.0516E-06 & 1.7618E-06 & 4.4039E-07 & 1.1007E-07 & 2.7470E-08 & $\approx$ 2.0024 \\
			0.3 & 4 & 7.0516E-06 & 1.7618E-06 & 4.4038E-07 & 1.1008E-07 & 2.7469E-08 & $\approx$ 2.0026 \\
			& 5 & 7.0516E-06 & 1.7618E-06 & 4.4039E-07 & 1.1009E-07 & 2.7578E-08 & $\approx$ 1.9971 \\
			& 6 & 7.0516E-06 & 1.7618E-06 & 4.4039E-07 & 1.1011E-07 & 2.7419E-08 & $\approx$ 2.0056 \\
			\hline
			& 2 & 2.0656E-06 & 5.1611E-07 & 1.2901E-07 & 3.2251E-08 & 8.0774E-09 & $\approx$ 1.9974 \\
			& 3 & 2.0656E-06 & 5.1611E-07 & 1.2901E-07 & 3.2251E-08 & 8.0778E-09 & $\approx$ 1.9973 \\
			0.8 & 4 & 2.0656E-06 & 5.1611E-07 & 1.2901E-07 & 3.2248E-08 & 8.0629E-09 & $\approx$ 1.9998 \\
			& 5 & 2.0656E-06 & 5.1611E-07 & 1.2901E-07 & 3.2248E-08 & 8.0769E-09 & $\approx$ 1.9973 \\
			& 6 & 2.0656E-06 & 5.1611E-07 & 1.2901E-07 & 3.2248E-08 & 8.0928E-09 & $\approx$ 1.9945 \\
			\hline
		\end{tabular}
		\label{Tab:homspaceL2}
	\end{table}
	
	\begin{table}[htbp]
		\caption{$H^{1}$-norm errors in space and convergence rates}
		\begin{tabular}{|c|c|ccccc|c|}
			\hline
			$\alpha$ & $k\backslash 1/h$ & 256 & 512 & 1024 & 2048 & 4096 & Rate \\
			\hline
			& 2 & 6.2678E-03 & 3.1320E-03 & 1.5657E-03 & 7.8284E-04 & 3.9142E-04 & $\approx$ 1.0000 \\
			& 3 & 6.2678E-03 & 3.1320E-03 & 1.5658E-03 & 7.8285E-04 & 3.9142E-04 & $\approx$ 1.0000 \\
			0.3 & 4 & 6.2678E-03 & 3.1320E-03 & 1.5658E-03 & 7.8285E-04 & 3.9142E-04 & $\approx$ 1.0000 \\
			& 5 & 6.2678E-03 & 3.1320E-03 & 1.5658E-03 & 7.8285E-04 & 3.9142E-04 & $\approx$ 1.0000 \\
			& 6 & 6.2678E-03 & 3.1320E-03 & 1.5658E-03 & 7.8285E-04 & 3.9142E-04 & $\approx$ 1.0000 \\
			\hline
			& 2 & 1.7994E-03 & 8.9919E-04 & 4.4953E-04 & 2.2476E-04 & 1.1238E-04 & $\approx$ 1.0000 \\
			& 3 & 1.7994E-03 & 8.9919E-04 & 4.4953E-04 & 2.2476E-04 & 1.1238E-04 & $\approx$ 1.0000 \\
			0.8 & 4 & 1.7994E-03 & 8.9919E-04 & 4.4953E-04 & 2.2476E-04 & 1.1238E-04 & $\approx$ 1.0000 \\
			& 5 & 1.7994E-03 & 8.9919E-04 & 4.4953E-04 & 2.2476E-04 & 1.1238E-04 & $\approx$ 1.0000 \\
			& 6 & 1.7994E-03 & 8.9919E-04 & 4.4953E-04 & 2.2476E-04 & 1.1238E-04 & $\approx$ 1.0000 \\
			\hline
		\end{tabular}
		\label{Tab:homspaceH1}
	\end{table}
	
	Furthermore, to show the effectiveness of our scheme and the significance of corrections for all steps, we provide another comparative example. Applying the correction scheme provided in \cite{jin2017} to our problem and taking $L^{2}$ projection on $e^{-t_{n}\rho U(x_{0})}G^{0}$, then one has the fully discrete scheme 
	\begin{equation}\label{eqfullhonw1}
	\left\{
	\begin{aligned}
	&\sum_{i=0}^{n-1}d^{\alpha,k}_{i}(e^{-t_{i}\rho U(x_{0})}G^{n-i}_{h} ,v_{h})+(A_{h} G^{n}_{h}, v_{h})+a^{(k)}_{n}(A_{h}P_{h}(e^{-t_{n}\rho U(x_{0})}G^{0}),v_{h})\\
	&\qquad=\sum_{i=0}^{n-1}d^{\alpha,k}_{i}(P_{h}(e^{-t_{n}\rho U(x_{0})}G^{0}),v_{h}) \qquad \forall v_{h}\in X_{h},\quad 1\leq n\leq k-1,\\
	&\sum_{i=0}^{n-1}d^{\alpha,k}_{i}(e^{-t_{i}\rho U(x_{0})}G^{n-i}_{h} ,v_{h})+(A_{h} G^{n}_{h}, v_{h})\\
	&\qquad=\sum_{i=0}^{n-1}d^{\alpha,k}_{i}(P_{h}(e^{-t_{n}\rho U(x_{0})}G^{0}),v_{h}) \qquad \forall v_{h}\in X_{h},\quad n\geq k,\\
	\end{aligned}\right.
	\end{equation}
	 And the source term $f$, initial data $G(x_{0},\rho,0)$, and $U(x_{0})$ are taken to be the same as the immediately above example. The corresponding $L^{2}$- and $H^{1}$-norm errors and convergence rates are presented in Tables \ref{Tab:homspaceL2w} and \ref{Tab:homspaceH1w}. It's easy to find that the convergence rates of $H^{1}$-norm errors are optimal but the convergence rates of $L^{2}$-norm errors can't achieve $O(h^{2})$. 

	\begin{table}[htbp]
		\caption{$L^{2}$-norm errors in space and convergence rates}
		\begin{tabular}{|c|c|ccccc|c|}
			\hline
			$\alpha$ & $k\backslash 1/h$ & 256 & 512 & 1024 & 2048 & 4096 & Rate \\
			\hline
			& 2 & 6.5612E-06 & 2.1429E-06 & 1.0074E-06 & 5.1449E-07 & 2.6272E-07 & $\approx$ 0.9696 \\
			& 3 & 1.4812E-05 & 7.5213E-06 & 3.8401E-06 & 1.9463E-06 & 9.8064E-07 & $\approx$ 0.9889 \\
			0.3 & 4 & 3.1889E-05 & 1.6237E-05 & 8.2196E-06 & 4.1387E-06 & 2.0772E-06 & $\approx$ 0.9945 \\
			& 5 & 5.4773E-05 & 2.7724E-05 & 1.3968E-05 & 7.0136E-06 & 3.5144E-06 & $\approx$ 0.9969 \\
			& 6 & 8.2820E-05 & 4.1748E-05 & 2.0980E-05 & 1.0520E-05 & 5.2679E-06 & $\approx$ 0.9978 \\
			\hline
			& 2 & 2.9559E-06 & 1.3904E-06 & 7.0643E-07 & 3.6011E-07 & 1.8223E-07 & $\approx$ 0.9827 \\
			& 3 & 1.0324E-05 & 5.2570E-06 & 2.6619E-06 & 1.3405E-06 & 6.7271E-07 & $\approx$ 0.9947 \\
			0.8 & 4 & 2.2209E-05 & 1.1224E-05 & 5.6481E-06 & 2.8338E-06 & 1.4195E-06 & $\approx$ 0.9974 \\
			& 5 & 3.7832E-05 & 1.9034E-05 & 9.5523E-06 & 4.7858E-06 & 2.3953E-06 & $\approx$ 0.9985 \\
			& 6 & 5.6851E-05 & 2.8532E-05 & 1.4299E-05 & 7.1590E-06 & 3.5818E-06 & $\approx$ 0.9991 \\
			\hline
		\end{tabular}
		\label{Tab:homspaceL2w}
	\end{table}

		\begin{table}[htbp]
		\caption{$H^{1}$-norm errors in space and convergence rates}
		\begin{tabular}{|c|c|ccccc|c|}
			\hline
			$\alpha$ & $k\backslash 1/h$ & 256 & 512 & 1024 & 2048 & 4096 & Rate \\
			\hline
			& 2 & 6.2678E-03 & 3.1320E-03 & 1.5658E-03 & 7.8285E-04 & 3.9142E-04 & $\approx$ 1.0000 \\
			& 3 & 6.2550E-03 & 3.1256E-03 & 1.5625E-03 & 7.8124E-04 & 3.9062E-04 & $\approx$ 1.0000 \\
			0.3 & 4 & 6.2357E-03 & 3.1159E-03 & 1.5577E-03 & 7.7883E-04 & 3.8941E-04 & $\approx$ 1.0000 \\
			& 5 & 6.2109E-03 & 3.1036E-03 & 1.5515E-03 & 7.7574E-04 & 3.8787E-04 & $\approx$ 1.0000 \\
			& 6 & 6.1818E-03 & 3.0890E-03 & 1.5443E-03 & 7.7210E-04 & 3.8604E-04 & $\approx$ 1.0000 \\
			\hline
			& 2 & 1.7995E-03 & 8.9920E-04 & 4.4953E-04 & 2.2476E-04 & 1.1238E-04 & $\approx$ 1.0000 \\
			& 3 & 1.7884E-03 & 8.9367E-04 & 4.4677E-04 & 2.2338E-04 & 1.1169E-04 & $\approx$ 1.0000 \\
			0.8 & 4 & 1.7722E-03 & 8.8558E-04 & 4.4272E-04 & 2.2135E-04 & 1.1068E-04 & $\approx$ 1.0000 \\
			& 5 & 1.7524E-03 & 8.7568E-04 & 4.3777E-04 & 2.1888E-04 & 1.0944E-04 & $\approx$ 1.0000 \\
			& 6 & 1.7306E-03 & 8.6475E-04 & 4.3231E-04 & 2.1615E-04 & 1.0807E-04 & $\approx$ 1.0000 \\
			\hline
		\end{tabular}
		\label{Tab:homspaceH1w}
	\end{table}
\end{example}

\begin{example}
	We take $\rho=-1+\mathbf{i}$,
	\begin{equation*}
	f(x_0,\rho,t)=x_{0}(1-x_{0})e^{-\rho\chi_{(0.5,1)}(x_{0})t},\quad G(x_{0},\rho,0)=0,~~{\rm and}~~U(x_{0})=\chi_{(0.5,1)}(x_{0}).
	\end{equation*}	
	We choose $\tau=1/200$ to decrease the influence caused by temporal discretizations. Use \eqref{eqfullscheme} to solve \eqref{eqretosol} and present the corresponding temporal errors and convergence rates in Tables \ref{Tab:NonspaceL2} and \ref{Tab:NonspaceH1}. All the  results agree with the predictions.
	\begin{table}[htbp]
		\caption{$L^{2}$-norm errors in space and convergence rates}
		\begin{tabular}{|c|c|ccccc|c|}
			\hline
			$\alpha$ & $k\backslash 1/h$ & 20 & 40 & 80 & 160 & 320 & Rate \\
			\hline
			& 2 & 3.3965E-05 & 8.4921E-06 & 2.1231E-06 & 5.3077E-07 & 1.3269E-07 & $\approx$ 2.0000 \\
			& 3 & 3.3964E-05 & 8.4919E-06 & 2.1230E-06 & 5.3076E-07 & 1.3269E-07 & $\approx$ 2.0000 \\
			0.3 & 4 & 3.3964E-05 & 8.4918E-06 & 2.1230E-06 & 5.3076E-07 & 1.3269E-07 & $\approx$ 2.0000 \\
			& 5 & 3.3964E-05 & 8.4918E-06 & 2.1230E-06 & 5.3076E-07 & 1.3269E-07 & $\approx$ 2.0000 \\
			& 6 & 3.3964E-05 & 8.4919E-06 & 2.1230E-06 & 5.3076E-07 & 1.3269E-07 & $\approx$ 2.0000 \\
			\hline
			& 2 & 3.3965E-05 & 8.4921E-06 & 2.1231E-06 & 5.3077E-07 & 1.3269E-07 & $\approx$ 2.0000 \\
			& 3 & 3.3964E-05 & 8.4919E-06 & 2.1230E-06 & 5.3076E-07 & 1.3269E-07 & $\approx$ 2.0000 \\
			0.6 & 4 & 3.3964E-05 & 8.4918E-06 & 2.1230E-06 & 5.3076E-07 & 1.3269E-07 & $\approx$ 2.0000 \\
			& 5 & 3.3964E-05 & 8.4918E-06 & 2.1230E-06 & 5.3076E-07 & 1.3269E-07 & $\approx$ 2.0000 \\
			& 6 & 3.3964E-05 & 8.4919E-06 & 2.1230E-06 & 5.3076E-07 & 1.3269E-07 & $\approx$ 2.0000 \\
			\hline
		\end{tabular}
		\label{Tab:NonspaceL2}
	\end{table}
	
	\begin{table}[htbp]
		\caption{$H^{1}$-norm errors in space and convergence rates}
		\begin{tabular}{|c|c|ccccc|c|}
			\hline
			$\alpha$ & $k\backslash 1/h$ & 20 & 40 & 80 & 160 & 320 & Rate \\
			\hline
			& 2 & 2.5503E-03 & 1.2756E-03 & 6.3784E-04 & 3.1893E-04 & 1.5947E-04 & $\approx$ 1.0000 \\
			& 3 & 2.5502E-03 & 1.2755E-03 & 6.3783E-04 & 3.1892E-04 & 1.5946E-04 & $\approx$ 1.0000 \\
			0.3 & 4 & 2.5502E-03 & 1.2755E-03 & 6.3783E-04 & 3.1892E-04 & 1.5946E-04 & $\approx$ 1.0000 \\
			& 5 & 2.5502E-03 & 1.2755E-03 & 6.3783E-04 & 3.1892E-04 & 1.5946E-04 & $\approx$ 1.0000 \\
			& 6 & 2.5502E-03 & 1.2755E-03 & 6.3783E-04 & 3.1892E-04 & 1.5946E-04 & $\approx$ 1.0000 \\
			\hline
			& 2 & 2.5503E-03 & 1.2756E-03 & 6.3784E-04 & 3.1893E-04 & 1.5947E-04 & $\approx$ 1.0000 \\
			& 3 & 2.5502E-03 & 1.2755E-03 & 6.3783E-04 & 3.1892E-04 & 1.5946E-04 & $\approx$ 1.0000 \\
			0.6 & 4 & 2.5502E-03 & 1.2755E-03 & 6.3783E-04 & 3.1892E-04 & 1.5946E-04 & $\approx$ 1.0000 \\
			& 5 & 2.5502E-03 & 1.2755E-03 & 6.3783E-04 & 3.1892E-04 & 1.5946E-04 & $\approx$ 1.0000 \\
			& 6 & 2.5502E-03 & 1.2755E-03 & 6.3783E-04 & 3.1892E-04 & 1.5946E-04 & $\approx$ 1.0000 \\
			\hline
		\end{tabular}
		\label{Tab:NonspaceH1}
	\end{table}
	
	\begin{table}[htbp]
		\caption{$L^{2}$-norm errors in space and convergence rates}
		\begin{tabular}{|c|c|ccccc|c|}
			\hline
			$\alpha$ & $k\backslash 1/h$ & 20 & 40 & 80 & 160 & 320 & Rate \\
			\hline
			& 2 & 9.0942E-06 & 6.0132E-06 & 5.5959E-06 & 3.3906E-06 & 1.8330E-06 & $\approx$ 0.8873 \\
			& 3 & 9.0858E-06 & 6.0174E-06 & 5.5979E-06 & 3.3916E-06 & 1.8335E-06 & $\approx$ 0.8874 \\
			0.3 & 4 & 9.0857E-06 & 6.0174E-06 & 5.5980E-06 & 3.3916E-06 & 1.8335E-06 & $\approx$ 0.8874 \\
			& 5 & 9.0857E-06 & 6.0174E-06 & 5.5980E-06 & 3.3916E-06 & 1.8335E-06 & $\approx$ 0.8874 \\
			& 6 & 9.0857E-06 & 6.0174E-06 & 5.5980E-06 & 3.3916E-06 & 1.8335E-06 & $\approx$ 0.8874 \\
			\hline
			& 2 & 9.0942E-06 & 6.0132E-06 & 5.5959E-06 & 3.3906E-06 & 1.8330E-06 & $\approx$ 0.8873 \\
			& 3 & 9.0858E-06 & 6.0174E-06 & 5.5979E-06 & 3.3916E-06 & 1.8335E-06 & $\approx$ 0.8874 \\
			0.6 & 4 & 9.0857E-06 & 6.0174E-06 & 5.5980E-06 & 3.3916E-06 & 1.8335E-06 & $\approx$ 0.8874 \\
			& 5 & 9.0857E-06 & 6.0174E-06 & 5.5980E-06 & 3.3916E-06 & 1.8335E-06 & $\approx$ 0.8874 \\
			& 6 & 9.0857E-06 & 6.0174E-06 & 5.5980E-06 & 3.3916E-06 & 1.8335E-06 & $\approx$ 0.8874 \\
			\hline
		\end{tabular}
		\label{Tab:NonspaceL2w}
	\end{table}
	
	\begin{table}[htbp]
		\caption{$H^{1}$-norm errors in space and convergence rates}
		\begin{tabular}{|c|c|ccccc|c|}
			\hline
			$\alpha$ & $k\backslash 1/h$ & 20 & 40 & 80 & 160 & 320 & Rate \\
			\hline
			& 2 & 2.5451E-03 & 1.2730E-03 & 6.3657E-04 & 3.1829E-04 & 1.5915E-04 & $\approx$ 1.0000 \\
			& 3 & 2.5450E-03 & 1.2730E-03 & 6.3655E-04 & 3.1828E-04 & 1.5914E-04 & $\approx$ 1.0000 \\
			0.3 & 4 & 2.5450E-03 & 1.2730E-03 & 6.3655E-04 & 3.1828E-04 & 1.5914E-04 & $\approx$ 1.0000 \\
			& 5 & 2.5450E-03 & 1.2730E-03 & 6.3655E-04 & 3.1828E-04 & 1.5914E-04 & $\approx$ 1.0000 \\
			& 6 & 2.5450E-03 & 1.2730E-03 & 6.3655E-04 & 3.1828E-04 & 1.5914E-04 & $\approx$ 1.0000 \\
			\hline
			& 2 & 2.5451E-03 & 1.2730E-03 & 6.3657E-04 & 3.1829E-04 & 1.5915E-04 & $\approx$ 1.0000 \\
			& 3 & 2.5450E-03 & 1.2730E-03 & 6.3655E-04 & 3.1828E-04 & 1.5914E-04 & $\approx$ 1.0000 \\
			0.6 & 4 & 2.5450E-03 & 1.2730E-03 & 6.3655E-04 & 3.1828E-04 & 1.5914E-04 & $\approx$ 1.0000 \\
			& 5 & 2.5450E-03 & 1.2730E-03 & 6.3655E-04 & 3.1828E-04 & 1.5914E-04 & $\approx$ 1.0000 \\
			& 6 & 2.5450E-03 & 1.2730E-03 & 6.3655E-04 & 3.1828E-04 & 1.5914E-04 & $\approx$ 1.0000 \\
			\hline
		\end{tabular}
		\label{Tab:NonspaceH1w}
	\end{table}
To show the differences between the two different projections, we also compute this example with the numerical scheme \eqref{eqfullnonwrong}, i.e., we take $L^{2}$ projection on $f$ directly,
	\begin{equation}\label{eqfullnonwrong}
	\begin{aligned}
	&\sum_{i=0}^{n-1}d^{\alpha,k}_{i}(e^{-t_{i}\rho U(x_{0})}G^{n-i}_{h} ,v_{h})+(A_{h} G^{n}_{h}, v_{h})=\sum_{i=0}^{n-1}d^{\alpha-1,k}_{i}(e^{-t_{i}\rho U(x_{0})}P_{h}f^{n-i},v_{h})\\
	&\qquad\quad+\sum_{j=1}^{k-1}a^{(k)}_{j}d^{\alpha-1,k}_{n-j}(e^{-t_{n-j}\rho U(x_{0})}P_{h}f^{0},v_{h})\\
	&\qquad\quad+\sum_{l=1}^{k-2}\sum_{j=1}^{k-1}b^{(k)}_{l,j}\tau^{l}d^{\alpha-1,k}_{n-j}(e^{-t_{n-j}\rho U(x_{0})}P_{h}\partial^{l}_{t}f(0),v_{h}) \qquad \forall v_{h}\in X_{h}.
	\end{aligned}
	\end{equation}
The relevant $L^{2}$- and $H^{1}$-norm errors and convergence rates are presented in Tables \ref{Tab:NonspaceL2w} and \ref{Tab:NonspaceH1w}, which show that the numerical scheme \eqref{eqfullnonwrong} delivers an $O(h)$ accuracy in $H^{1}$-norm and only about $O(h^{0.9})$ accuracy in $L^{2}$-norm. Thus, according to the results in Tables \ref{Tab:NonspaceL2} and \ref{Tab:NonspaceL2w}, our schemes can solve Eq. \eqref{eqretosol} more effectively.
	
\end{example}

\section{Conclusions}

Functional, as an important class of statistical observables, plays a key role in uncovering the mechanism of anomalous dynamics and extending their applications. The probability density function of the functional for anomalous dynamics is governed by fractional Feynman-Kac equation. The time-space coupled operator of the equation, the possible low regularity of the functional, and the complex variables bring the challenges in effectively solving the equation. This paper carefully designs the numerical schemes, which can achieve the optimal time convergence rates up to order $6$ and optical space convergence rate without any regularity assumptions on the solution. The convergence results are theoretically proved and numerically confirmed. More numerical experiments are also performed to show the benefits of the schemes presented in this paper comparing with the existing ones in solving the fractional Feynman-Kac equation.


\bibliographystyle{siamplain}
\bibliography{references}

\begin{thebibliography}{10}

\bibitem{Carmi2010}
{\sc S.~Carmi, L.~Turgeman, and E.~Barkai}, {\em On distributions of
  functionals of anomalous diffusion paths}, J. Stat. Phys., 141 (2010),
  pp.~1071--1092.

\bibitem{Chenminghua:15}
{\sc M.~Chen and W.~Deng}, {\em High order algorithms for the fractional
  substantial diffusion equation with truncated {L}{\'e}vy flights}, SIAM J.
  Sci. Comput., 37 (2015), pp.~A890--A917.

\bibitem{ChenDeng:18}
{\sc M.~Chen and W.~Deng}, {\em High order algorithm for the time-tempered
  fractional {Feynman}-{Kac} equation}, J. Sci. Comput., 76 (2018), pp.~1--21.

\bibitem{chen2015}
{\sc S.~Chen, J.~Shen, and L.~Wang}, {\em Generalized {Jacobi} functions and
  their applications to fractional differential equations}, Math. Comp., 85
  (2015), pp.~1603--1638.

\bibitem{creedon:1975}
{\sc D.~M. Creedon and J.~J.~H. Miller}, {\em The stability properties of
  q-step backward difference schemes}, BIT, 15 (1975), pp.~244--249.

\bibitem{Deng2020}
{\sc W.~Deng, R.~Hou, W.~Wang, and P.~Xu}, {\em Modeling Anomalous Diffusion:
  From Statistics to Mathematics}, World Scientific, Singapore, 2020.

\bibitem{DengLiQianWang:18}
{\sc W.~Deng, B.~Li, Z.~Qian, and H.~Wang}, {\em Time discretization of a
  tempered fractional {Feynman}-{Kac} equation with measure data}, SIAM J.
  Numer. Anal., 56 (2018), pp.~3249--3275.

\bibitem{ford2017}
{\sc N.~J. Ford and Y.~Yan}, {\em An approach to construct higher order time
  discretisation schemes for time fractional partial differential equations
  with nonsmooth data}, Fract. Calc. Appl. Anal., 20 (2017), pp.~1076--1105.

\bibitem{Jin2016}
{\sc B.~Jin, R.~Lazarov, and Z.~Zhou}, {\em Two {fully} {discrete} {schemes}
  for {fractional} {diffusion} and {diffusion}-{wave} {equations} with
  {nonsmooth} {data}}, SIAM J. Sci. Comput., 38 (2016), pp.~A146--A170.

\bibitem{jin2017}
{\sc B.~Jin, B.~Li, and Z.~Zhou}, {\em Correction of {high}-{order} {BDF}
  {convolution} {quadrature} for {fractional} {evolution} {equations}}, SIAM J.
  Sci. Comput., 39 (2017), pp.~A3129--A3152.

\bibitem{Kac:1949}
{\sc M.~Kac}, {\em On distributions of certain {W}iener functionals}, Trans.
  Amer. Math. Soc., 65 (1949), pp.~1--13.

\bibitem{li:2020}
{\sc B.~Li, K.~Wang, and Z.~Zhou}, {\em Long-time {accurate} {symmetrized}
  {implicit}-explicit {BDF} {methods} for a {class} of {parabolic} {equations}
  with {non}-self-adjoint {operators}}, SIAM J. Numer. Anal., 58 (2020),
  pp.~189--210.

\bibitem{LiDengZhao:19}
{\sc C.~Li, W.~Deng, and L.~Zhao}, {\em Well-posedness and numerical algorithm
  for the tempered fractional differential equations}, Discrete Contin. Dyn.
  Syst. Ser. B, 24 (2019), pp.~1989--2015.

\bibitem{lubich1988-1}
{\sc C.~Lubich}, {\em Convolution {quadrature} and {discretized} {operational}
  {calculus} {I}.}, Numer. Math., 52 (1988), pp.~129--145.

\bibitem{lubich1988-2}
{\sc C.~Lubich}, {\em Convolution quadrature and discretized operational
  calculus. {II}}, Numer. Math., 52 (1988), pp.~413--425.

\bibitem{lubich1996}
{\sc C.~Lubich, I.~H. Sloan, and V.~Thom{\'e}e}, {\em Nonsmooth {data} {error}
  {estimates} for {approximations} of an {evolution} {equation} with a
  {positive}-{type} {memory} {term}}, Math. Comp., 65 (1996), pp.~1--17.

\bibitem{Podlubny1999}
{\sc I.~Podlubny}, {\em Fractional Differential Equations}, Academic Press, San
  Diego, 1999.

\bibitem{Sun:2020}
{\sc J.~Sun, D.~Nie, and W.~Deng}, {\em Error estimates for backward fractional
  {Feynman}-{Kac} equation with non-smooth initial data}, J. Sci. Comput., 84
  (2020), p.~6.

\bibitem{Thomee2006}
{\sc V.~Thom{\'e}e}, {\em Galerkin Finite Element Methods for Parabolic
  Problems}, Springer Berlin Heidelberg, 2nd~ed., 1997.

\bibitem{TurgemanCarmiBarkai:09}
{\sc L.~Turgeman, S.~Carmi, and E.~Barkai}, {\em Fractional {Feynman}-{Kac}
  {equation} for {Non}-{Brownian} {functionals}}, Phys. Rev. Lett., 103 (2009),
  p.~190201.

\bibitem{yan2018}
{\sc Y.~Yan, M.~Khan, and N.~J. Ford}, {\em An {analysis} of the {modified}
  {L1} {scheme} for {time}-{fractional} {partial} {differential} {equations}
  with {nonsmooth} {data}}, SIAM J. Numer. Anal., 56 (2018), pp.~210--227.

\bibitem{ZhangDeng:17}
{\sc Z.~Zhang and W.~Deng}, {\em Numerical approaches to the functional
  distribution of anomalous diffusion with both traps and flights}, Adv.
  Comput. Math., 43 (2017), pp.~699--732.

\end{thebibliography}

\end{document}